\documentclass{amsart}

\usepackage{mathrsfs}
\usepackage{amssymb}
\usepackage{hyperref}
\hypersetup{
	colorlinks=true,
	linkcolor=blue,
	filecolor=blue,      
	urlcolor=blue,
	citecolor=cyan,
}
\usepackage{booktabs} 
\usepackage{array}
\usepackage{float}
\usepackage {color}
\newtheorem{theorem}{Theorem}[section]
\newtheorem{lemma}[theorem]{Lemma}

\theoremstyle{definition}
\newtheorem{definition}[theorem]{Definition}
\newtheorem{proposition}[theorem]{Proposition}
\newtheorem{corollary}[theorem]{Corollary}

\theoremstyle{remark}

\numberwithin{equation}{section}

\newcommand{\abs}[1]{\lvert#1\rvert}
\newcommand{\Abs}[1]{\Big|#1\Big|}
\newcommand{\norm}[1]{\lVert#1\rVert}
\newcommand{\Norm}[1]{\Big\lVert#1\Big\rVert}
\newcommand{\lr}[1]{\left(#1\right)}
\newcommand{\Lr}[1]{\left[#1\right]}
\newcommand{\LR}[1]{\left\{#1\right\}}

\newcommand{\FT}[2]{\mathscr{F}_{#1}\lr{#2}}

\newcommand{\supp}{\mathbf{supp}}
\newcommand{\ext}{\mathbf{ext}}

\newcommand{\innerproduct}[2]{\langle #1, #2 \rangle}

\begin{document}
	
	\title{Paley-Wiener theorems for slice monogenic functions}
	
	\author{Yanshuai Hao}
	\address{School of Computer Science and Engineering, Faculty of Innovation Engineering, Macau University of Science and Technology}
	\curraddr{}
	\email{3220004726@student.must.edu.mo}
	\thanks{}
	
	%    author two information
	\author{Pei Dang}
	\address{Department of Engneering Science, Faculty of Innovation Engineering, Macau University of Science and Technology}
	\curraddr{}
	\email{pdang@must.edu.mo}
	\thanks{}
	
	%    author three information
	\author{Weixiong Mai$^{\star}$}
	\address{School of Computer Science and Engineering, Faculty of Innovation Engineering, Macau University of Science and Technology}
	\curraddr{}
	\email{wxmai@must.edu.mo}
	\thanks{$^{\star}$Corresponding author}
	
	%    General info
	% 30G35 Functions of hypercomplex variables and generalized variables
	% 32A26 Integral representations, constructed kernels (e.g. Cauchy, Fantappi`etype kernels)
	% 42A38 Fourier and Fourier-Stieltjes transforms and other transforms of Fourier type
	% 15A66 Clifford algebras, spinors
	
	\date{}
	
	\dedicatory{}
	
	
	\begin{abstract}
		In this paper, we prove some Paley-Wiener theorems for function spaces consisting of  slice monogenic functions such as Paley-Wiener, Hardy and Bergman spaces. As applications, we can compute the reproducing kernel functions for the related function spaces. 
	\end{abstract}
	
	\maketitle
	\section{Introduction}
	
	The Paley-Wiener theorem plays an important role in complex analysis. It shows the connection between  holomorphic functions $f$ in  certain function spaces and the Fourier transform of $f$. The classical Paley-Wiener theorem has two different forms. One is in regard to the Hardy space on the upper half-plane $H^2(\mathbb C_+)$, and the other is corresponding to a function space consisting of entire functions of exponential type called the Paley-Wiener space. The former one says that if $f\in L^{2}(\mathbb{R})$, then $\mathscr{F}\left(f\right)$ is supported on $[0,\infty)$ if and only if $f$ is the non-tangential boundary limit (NTBL) of some function in $H^{2}(\mathbb{C}_{+})$, where $\mathscr{F}\left(f\right)$ denotes the Fourier transform of $f$. The other one states that for given positive constants $B$ and $C$,  and $f$ an entire function satisfying $f|_\mathbb R \in L^2\left(\mathbb{R}\right)$, there holds that the support of $\mathscr{F}\left(f|_{\mathbb{R}}\right)$ lies on $\Lr{-B,B}$ if and only if $\abs{f\left(z\right)}\leq Ce^{B|z|}$, where $f|_{\mathbb{R}}$ is the restriction of $f$ on $\mathbb{R}$ (\cite{paley1934fourier}). For convenience, in this paper we refer to the two forms as the non-compact type and the compact type Paley-Wiener theorems, respectively.
	
	Since the classical Paley-Wiener theorem was proposed in 1934, a lot of attention has been paid to Paley-Wiener theorem in different settings. For instance, the Paley-Wiener type theorems have been extensively studied including generalizations in the $L^p$-sense,  distribution sense, analogous results for Paley-Wiener spaces, Hardy spaces, Bergman spaces and Dirichlet spaces in one and several variables (see e.g. \cite{GENCHEV1986496,bernstein1998paley,duren2007paley,qian2005characterization,gilbert1991clifford,qian2009fourier,li2018fourier,dang2020fourier,os2001generalized,hormander2015analysis,schwartz_transformation_1952}). In this paper we are mainly concerned with the compact and non-compact type Paley-Wiener theorems, so in the following we give the main developments of those two types of Paley-Wiener theorem. In \cite{qian2005characterization,qian2009fourier}, the authors generalize the classical non-compact type Paley-Wiener theorem to the case of $1\leq p\leq\infty$. Specifically, it is shown that for $f\in L^{p}(\mathbb{R})$, $1\leq p\leq\infty$,  $\lr{\mathscr{F}\lr{f},\varphi}=0$  if and only if $f$ is the NTBL of some function in $H^{p}\lr{\mathbb{C}_{+}}$, where $\varphi$ belongs to Schwarz space $\mathscr{S}$ with support lying on $(-\infty,0]$. In \cite{boas1938representations}, the authors generalize of the classical compact type Paley-Wiener theorem to the case of $1<p\leq2$. It is stated that if an entire function $f$ satisfies $\abs{f\lr{z}}<Ce^{B\abs{z}}$ and $f|_{\mathbb{R}}\in L^{p}(\mathbb{R})$ for $1<p\leq2$, then $f\lr{z}=\int_{-B}^{B}e^{izt}\mathscr{F}\lr{f|_{\mathbb{R}}}\lr{t}\mathrm{d}t$ where $\mathscr{F}\lr{f|_{\mathbb{R}}}\in L^{\frac{p}{p-1}}\lr{-B,B}$. Conversely, if $f\lr{z}=\int_{-B}^{B}e^{izt}\mathscr{F}\lr{f|_{\mathbb{R}}}\lr{t}\mathrm{d}t$ with $\mathscr{F}\lr{f|_{\mathbb{R}}}\in L^{p}\lr{-B,B}$, then $f$ is an entire function satisfying $\abs{f\lr{z}}<Ce^{B\abs{z}}$ and $f|_{\mathbb{R}}\in L^{\frac{p}{p-1}}(\mathbb{R})$. Besides, the Paley-Wiener theorem is generalized to several variables in the settings of several complex variables and Clifford algebra, respectively. In the several complex variables setting, the non-compact type Paley-Wiener theorem gives a Fourier spectrum characterization of Hardy $H^2$ space on tubes over regular cone (see e.g. \cite{stein1971introduction}). Moreover, the result is,  respectively,  generalized to Hardy $H^p$  spaces on tubes over regular cone with $1\leq p\leq \infty$ (\cite{li2018fourier}) and $0<p<1$ (\cite{deng_fourier_2019}). The compact type case in the setting of several complex variables is given in e.g. \cite{stein1971introduction}. In the Clifford algebra setting, the non-compact type Paley-Wiener theorem is respectively studied in \cite{bernstein1998paley} and \cite{gilbert1991clifford} for $p=2$, and in \cite{dang2020fourier} a systematical investigation of Paley-Wiener theorem for $1\leq p\leq\infty$ is given. Note that the compact type Paley-Wiener theorem in the Clifford algebra setting is proved in \cite{kou2002paley} by a substantial method.
	
	It is noted that slice monogenic functions were first introduced in \cite{Colombo2009Slice}, which is a generalization of slice regular functions of a quaternionic variable initiated by Gentili and Struppa \cite{gentili2006new}.  It generalizes the theory of holomorphic functions of one complex variable to higher dimensional spaces, which is different from the holomorphic function theory of several complex variables and monogenic function theory of Clifford algebras. Over the past two decades, slice analysis including the cases of quaternion and Clifford algebras has been widely studied.  It has been significantly developed in many disciplines including geometric function theory, operator theory, differential geometry, etc. (see \cite{colombo2011cauchy,colombo2014some,colombo2022poisson,colombo2021slice,colombo2013nonconstant,ren2017growth,colombo2014distributions,jin2023adaptive,ren2017julia,colombo2011runge,dou2022extension}).  
	
	The classical Paley-Wiener theorem plays a role in the theory of holomorphic functions of one complex variable.  Therefore, it would be interesting and significant to investigate the analogous results of Paley-Wiener theorem for slice monogenic functions. In fact, in the setting of quaternion analysis, we proved the Paley-Wiener theorems for Hardy space and Paley-Wiener space consisting of slice regular functions (\cite{Hao-Dang-Mai}). This paper could be considered as a continuous study of \cite{Hao-Dang-Mai}. 
	In this paper, we will prove, respectively, the non-compact type and compact type of the Paley-Wiener theorem for slice monogenic functions.  In the non-compact type case, we will obtain the related results for  $1\leq p\leq 2.$ In the compact case, we only obtain the related results for $p=2.$ To do so, we need to introduce the one-dimensional left-sided Clifford Fourier transform (1DLCFT) and the Hardy and Bergman spaces consisting of slice monogenic functions. Moreover, we can exactly compute the formulas of related reproducing kernel functions by the obtained Paley-Wiener type theorems. In fact, it is easy to prove the Paley-Wiener type theorems for slice monogenic functions in a particular slice since the function theory in such case is very similar to the holomorphic function theory of one complex variable. However, the main difficulty is to generalize the related results to the whole space.
	
	The paper is organized as follows. In \S 2 we introduce some notations, definitions and properties for Clifford algebra and slice monogenic functions. In \S 3 we prove three Paley-Wiener theorem for slice monogenic Paley-Wiener, Hardy and Bergman spaces, respectively. 
	
	\section{Preliminaries}
	
	In this section we introduce some notations and fundamental properties of Clifford algebras and slice monogenic functions. For more information, please refer to \cite{brackx1982clifford,Colombo2009Slice,delanghe2012clifford,sabadini2011noncommutative}.
	
	Let $\mathbb{R}_n$ be the real Clifford algebra over $n$ imaginary units $\mathbf{e}_1,\dots,\mathbf{e}_n$ satisfying $\mathbf{e}_{r}\mathbf{e}_{s}+\mathbf{e}_{s}\mathbf{e}_{r}=-2\delta_{rs}$, where $\delta_{rs}$ is the Kronecker delta. Denote the ordered powerset of $X$ by
	$$\mathcal{P}_{<}\lr{X}:=\Big\{\LR{i_1,\dots,i_s}\in\mathcal{P}\lr{X}\setminus\LR{\varnothing}:i_1<\dots<i_s\Big\}\cup\LR{\varnothing}$$
	where $\mathcal{P}\lr{X}$ is the powerset of $X$. An element in $\mathbb{R}_{n}$, called a Clifford number, can be written as $a=\sum_{A}a_{A}\mathbf{e}_{A}$, where $A\in\mathcal{P}_{<}\lr{\LR{1,\dots,n}}$ and 
	\begin{equation*}
		\mathbf{e}_{A}=
		\begin{cases}
			1,  &A=\varnothing,\\
			\mathbf{e}_{i_1}\mathbf{e}_{i_2}\dots \mathbf{e}_{i_s}, &A=\LR{i_1,\dots,i_s}.
		\end{cases}
	\end{equation*}
	Let 
	\[\mathbb{R}_{n}^{1}=\left\{\underline{x}=x_1\mathbf{e}_1+\dots+x_n\mathbf{e}_n:x_i\in\mathbb{R},i=1,2,\dots,n\right\}\]
	be identical with the usual Euclidean space $\mathbb{R}^{n}$, and 
	\[\mathbb{R}^{n+1}=\left\{\mathbf{x}=x_0+\underline{x}:x_0\in\mathbb{R},\underline{x}\in\mathbb{R}_{n}^{1}\right\}.\]
	An element in $\mathbb{R}^{n+1}$ is a paravector. The conjugation of $\textbf{x}\in\mathbb{R}^{n+1}$ is given by
	\[\overline{\textbf{x}}=x_0-x_1\textbf{e}_1-\dots x_n\textbf{e}_n=x_0-\underline{x}\]
	and the norm of $\textbf{x}$ is
	\[\abs{\textbf{x}}=\sqrt{x_0^2+x_1^2+\dots+x_n^2}=\sqrt{x_0^2+\abs{\underline{x}}^2}.\]
	The inverse of $\textbf{x}$ is defined by
	\[\textbf{x}^{-1}=\frac{\overline{\textbf{x}}}{\abs{\textbf{x}}^{2}}.\]
	We denote by $\mathbb{S}$ the sphere
	\[\mathbb{S}=\left\{\underline{x}=x_1\mathbf{e}_1+\dots+x_n\mathbf{e}_n\in\mathbb{R}^{n+1}:x_1^2+\dots+x_n^2=1\right\}.\]
	Note that for $I\in\mathbb{S}$, we obviously have $I^2=-1$. The vector space $\mathbb{R}+I\mathbb{R}$ passing through $1$ and $I$ will be denoted by $L_{I}$. In fact, it is isomorphic to the complex plane. An element $\mathbf{x}\in L_{I}$ will be denoted by $u+Iv$, for $u,v\in\mathbb{R}$. To any non real paravector $\textbf{x}\in\mathbb{R}^{n+1}$, we can uniquely associate the element in the sphere $\mathbb{S}$ given by $I_{\textbf{x}}:=\frac{\underline{x}}{\abs{\underline{x}}}$, and write $\textbf{x}=u+I_{\textbf{x}}v$, where $u=x_0$, $v=\abs{x}$. 
	
	Given an element $\mathbf{x}=x_0+\underline{x}\in\mathbb{R}^{n+1}$, we define
	\[\Lr{\mathbf{x}}:=\LR{\mathbf{y}\in\mathbb{R}^{n+1}:\mathbf{y}=x_0+I\abs{\underline{x}}, I\in\mathbb{S}}.\]
	The set $\Lr{\mathbf{x}}$ is a $\lr{n-1}$-dimensional sphere in $\mathbb{R}^{n+1}$.
	%Definition
	\begin{definition}[\cite{Colombo2009Slice}]
		Let $U\subseteq\mathbb{R}^{n+1}$ be an open set and let $f:U\to\mathbb{R}_n$ be a real differentiable function. Let $I\in\mathbb{S}$ and let $f_I$ be the restriction of $f$ to the complex plane $L_{I}=\mathbb{R}+I\mathbb{R}$. Denote by $u+Iv$ an element on $L_{I}$. We say that $f$ is a left slice monogenic function if, for every $I\in\mathbb{S}$, we have 
		\[\overline{\partial_{I}}f\left(u+Iv\right):=\frac{1}{2}\left(\frac{\partial}{\partial u}+I\frac{\partial}{\partial v}\right)f_{I}\left(u+Iv\right)=0\]
		on $U\cap L_{I}$. Denote by $\mathcal{SM}_{l}\left(U\right)$ the set of left slice monogenic functions on the open set $U$. Similarly, a function is said to be right slice monogenic in $U$ if for every $I\in\mathbb{S}$, 
		\[\frac{1}{2}\left(\frac{\partial}{\partial u}f_{I}\left(u+Iv\right)+\frac{\partial}{\partial v}f_{I}\left(u+Iv\right)I\right)=0\]
		on $U\cap L_{I}$.
	\end{definition}
	In this paper, we only discuss the case for left slice monogenic functions. Note that for the case of right slice monogenic functions we can have similar results.
	\begin{definition}[\cite{Colombo2009Slice}]
		Let $U\subseteq\mathbb{R}^{n+1}$ be a domain. We say that $U$ is a slice domain (s-domain for short) if $U\cap\mathbb{R}$ is nonempty and if $U\cap L_{I}$ is a domain in $L_{I}$ for all $I\in\mathbb{S}$.
	\end{definition}
	\begin{definition} [\cite{Colombo2009Slice}]
		Let $U\subseteq\mathbb{R}^{n+1}$. We say that $U$ is axially symmetric if, for all $\mathbf{x}\in U$, the whole $(n-1)$-sphere $[\mathbf{x}]$ is contained in $U$.
	\end{definition}
	\begin{definition}[\cite{MR4752422}]
		Let $f:\mathbb{R}\to\mathbb{R}_{n}$ be a Clifford-valued function with $f\in L^{1}(\mathbb{R})$ and let $I\in\mathbb{S}$. Then the one-dimensional left-sided Clifford Fourier transform (1DLCFT) $\mathscr{F}_I(f)$ of $f$ is 
		\begin{equation*}
			\mathscr{F}_I(f)(w):=\frac{1}{\sqrt{2\pi}}\int_{\mathbb{R}}e^{-Iuw}f(u)\mathrm{d}u.
		\end{equation*}
	\end{definition}
	Some fundamental properties of the 1DLCFT are given as follows (see, \cite{MR4752422}).
	\begin{itemize}
		%\item Let $a,b\in\mathbb{R}_{n}$ and $f,g\in L^{1}(\mathbb{R})$. Then
		%\begin{equation*}
		%    \FT{I}{fa+gb}\lr{w}=\FT{I}{f}\lr{w}a+\FT{I}{g}\lr{w}b.
		%\end{equation*}
		%\item Let $f\in L^{1}(\mathbb{R})$ and $h\in\mathbb{R}$. Then
		%\begin{equation*}
		%    \FT{I}{f\lr{u+h}}\lr{w}=e^{Iwh}\FT{I}{f}\lr{w},
		%\end{equation*}
		%and
		%\begin{equation*}
		%    \FT{I}{f}\lr{w+h}=\FT{I}{e^{-Iuh}f\lr{u}}\lr{w}.
		%\end{equation*}
		\item (Inverse Fourier transform) Let $f,\FT{I}{f}\in L^{1}(\mathbb{R})$. Then the inverse Fourier transform is defined by
		\begin{equation*}
			f\lr{u}=\mathscr{F}_{I}^{-1}\lr{\FT{I}{f}}\lr{u}=\frac{1}{\sqrt{2\pi}}\int_{\mathbb{R}}e^{Iuw}\FT{I}{f}(w)\mathrm{d}w.
		\end{equation*}
		\item (Plancherel's theorem) Let $f\in L^{2}(\mathbb{R})$. Then $$\lVert\FT{I}{f}\rVert_{L^{2}(\mathbb{R})}=\lVert f\rVert_{L^{2}(\mathbb{R})}.$$
		\item (Hausdorff-Young inequality) Let $f\in L^{p}(\mathbb{R})$, $1\leq p\leq2$. Then $$\lVert\FT{I}{f}\rVert_{L^{q}(\mathbb{R})}\leq\lVert f\rVert_{L^{p}(\mathbb{R})}$$
		where $\frac{1}{p}+\frac{1}{q}=1$.
	\end{itemize} 
	
	%Propositin
	\begin{proposition}[\cite{Colombo2009Slice}]\label{prop-I}
		Let $I=I_1\in\mathbb{S}$. It is possible to choose $I_2,\dots,I_n\in\mathbb{S}$ such that $I_1,\dots,I_n$ form an orthonormal basis for the Clifford algebra $\mathbb{R}_n$ i.e., they satisfy the defining relations $I_rI_s+I_sI_r=-2\delta_{rs}$.
	\end{proposition}
	\begin{proposition}[\cite{Colombo2009Slice}, Splitting Lemma]\label{prop-SL}
		Let $U\subseteq\mathbb{R}^{n+1}$ be an open set. Let $f:U\to\mathbb{R}_{n}$ be a left slice monogenic function. For every $I=I_1\in\mathbb{S}$, let $I_2,\dots,I_n$ be a completion to a basis of $\mathbb{R}_n$ satisfying the defining relations $I_rI_s+I_sI_r=-2\delta_{rs}$. Then there exist $2^{n-1}$ holomorphic functions $F_{A'}:U\cap L_{I}\to L_{I}$ such that for every $z=u+Iv$,
		\[f_I\left(z\right)=\sum_{A'}F_{A'}\left(z\right)I_{A'},\]
		where $A'\in\mathcal{P}_{<}\lr{\LR{2,\dots,n}}$ and 
		\begin{equation*}
			I_{A'}=
			\begin{cases}
				1,  &A'=\varnothing,\\
				I_{i_1}I_{i_2}\dots I_{i_s}, &A'=\LR{i_1,\dots,i_s}.
			\end{cases}
		\end{equation*}
	\end{proposition}
	
	Since $F_{A'}$ is $L_{I}$-valued, we can define $\Re\LR{F_{A'}}=\frac{1}{2}\lr{F_{A'}+\overline{F_{A'}}}$ and $\Im\LR{F_{A'}}=\frac{I}{2}\lr{\overline{F_{A'}}-F_{A'}}$. It is obvious that $\Re\LR{F_{A'}}$ and $\Im\LR{F_{A'}}$ are real-valued. Notice that
	\begin{align*}
		f\left(u+Iv\right)=&\sum_{A'}F_{A'}\left(z\right)I_{A'}\\
		=&\sum_{A'}\lr{\Re\LR{F_{A'}\left(u+Iv\right)}+I\Im\LR{F_{A'}\left(u+Iv\right)}}I_{A'}\\
		=&\Re\LR{F_{\varnothing}\left(u+Iv\right)}\\
		&+\sum_{k=1}^{n-1}\Lr{\sum_{\abs{A'}=k}\Re\LR{F_{A'}\left(u+Iv\right)}I_{A'}+\sum_{\abs{A'}=k-1}\Im\LR{F_{A'}\left(u+Iv\right)}II_{A'}}\\
		&+\Im\LR{F_{2\dots n}\left(u+Iv\right)}II_{2\dots n}.
	\end{align*}
	Using the fact that $I=I_{1},I_{2},\dots,I_{n}$ form an orthonormal basis of $\mathbb{R}_{n}$, we have 
	\begin{equation*}
		\abs{f\left(u+Iv\right)}^{2}=\sum_{A'}\lr{\abs{\Re\LR{F_{A'}\lr{u+Iv}}}^{2}+\abs{\Im\LR{F_{A'}\lr{u+Iv}}}^{2}}=\sum_{A'}\abs{F_{A'}\left(u+Iv\right)}^{2}.
	\end{equation*}
	\begin{proposition}[\cite{colombo2009structure}, Representation Formula]\label{prop-RF}
		Let $U\subseteq\mathbb{R}^{n+1}$ be an axially symmetric s-domain and let $f$ be a left slice monogenic function on $U$.  For every vector $\mathbf{x}\in U$ the following formulas hold for every $I\in\mathbb{S}$:
		\begin{equation*}
			f\left(\mathbf{x}\right)=\frac{1-I_{\mathbf{x}}I}{2}f\left(u+Iv\right)+\frac{1+I_{\mathbf{x}}I}{2}f\left(u-Iv\right)
		\end{equation*}
		and
		\begin{equation*}
			f\left(\mathbf{x}\right)=\frac{1}{2}\left[f\left(u-Iv\right)+f\left(u+Iv\right)\right]+\frac{I_{\mathbf{x}}I}{2}\left[f\left(u-Iv\right)-f\left(u+Iv\right)\right].
		\end{equation*}
		Moreover, the two quantities
		\begin{equation*}
			\alpha\left(u,v\right):=\frac{1}{2}\left[f\left(u-Iv\right)+f\left(u+Iv\right)\right]
		\end{equation*}
		and
		\begin{equation*}
			\beta\left(u,v\right):=\frac{I}{2}\left[f\left(u-Iv\right)-f\left(u+Iv\right)\right]
		\end{equation*}
		do not depend on $I\in\mathbb{S}$.
	\end{proposition}
	\begin{proposition}[\cite{10.1063/1.3636718}]\label{prop-norm-alpha beta}
		Let $U\subseteq\mathbb{R}^{n+1}$ be an axially symmetric s-domain, let $D\subseteq\mathbb{R}^2$ be such that $\mathbf{x}=u+Iv$ for all $\lr{u,v}\in D$ and let $f:U\to\mathbb{R}_n$. The function $f$ is left slice monogenic if and only if there exist two differentiable functions $\alpha,\beta:D\to\mathbb{R}_n$ satisfying $\alpha\lr{u,v}=\alpha\lr{u,-v}$, $\beta\lr{u,v}=-\beta\lr{u,-v}$ and the Cauchy-Riemann system
		\begin{equation*}
			\begin{cases}
				\partial_{u}\alpha-\partial_{v}\beta=0\\
				\partial_{v}\alpha+\partial_{u}\beta=0
			\end{cases}
		\end{equation*}
		such that $f\left(u+Iv\right)=\alpha(u,v)+I\beta(u,v)$.
	\end{proposition}
	\begin{proposition}[\cite{10.1063/1.3636718}]
		Let $U\subseteq\mathbb{R}^{n+1}$ be an axially symmetric s-domain and let $f$ be a left slice monogenic function on $U$. For every $I=I_1\in\mathbb{S}$, let $I_2,\dots,I_n$ be a completion to a basis of $\mathbb{R}_n$ satisfying the defining relations $I_rI_s+I_sI_r=-2\delta_{rs}$. Then there exist $2^{n}$ holomorphic intrinsic functions $h_{A}:U\cap L_{I}\to L_{I}$, such that $$f_{I}\left(z\right)=\sum_{A}h_{A}\left(z\right)I_{A},$$ where $A\in\mathcal{P}_{<}\lr{\LR{1,\dots,n}}$ and 
		\begin{equation*}
			I_{A}=
			\begin{cases}
				1,  &A=\varnothing,\\
				I_{i_1}I_{i_2}\dots I_{i_s}, &A=\LR{i_1,\dots,i_s}.
			\end{cases}
		\end{equation*}
	\end{proposition}
	Since $h_{A}\lr{u+Iv}$ is $L_{I}$-valued, we have 
	$$h_{A}\lr{u+Iv}=\Re\LR{h_{A}\lr{u+Iv}}+I\Im\LR{h_{A}\lr{u+Iv}}.$$
	By $\overline{h_{A}\lr{u+Iv}}=h_{A}\lr{u-Iv}$, we get $\Re\LR{h_{A}\lr{u+Iv}}=\Re\LR{h_{A}\lr{u-Iv}}$ and $\Im\LR{h_{A}\lr{u+Iv}}=-\Im\LR{h_{A}\lr{u-Iv}}$.
	Note that
	\begin{align*}
		f\lr{u+Iv}=&\sum_{A}h_{A}\lr{u+Iv}I_{A}\\
		=&\sum_{A}\Lr{\Re\LR{h_{A}\lr{u+Iv}}+I\Im\LR{h_{A}\lr{u+Iv}}}I_{A}\\
		=&\sum_{A}\Re\LR{h_{A}\lr{u+Iv}}I_{A}+I\sum_{A}\Im\LR{h_{A}\lr{u+Iv}}I_{A}\\
		=&\alpha\lr{u,v}+I\beta\lr{u,v}
	\end{align*}
	where $\alpha\lr{u,v}=\sum_{A}\Re\LR{h_{A}\lr{u+Iv}}I_{A}$ and $\beta\lr{u,v}=\sum_{A}\Im\LR{h_{A}\lr{u+Iv}}I_{A}$. It is obvious that $\alpha\lr{u,v}=\alpha\lr{u,-v}$, $\beta\lr{u,v}=-\beta\lr{u,-v}$ and 
	\begin{equation*}
		\begin{cases}
			\partial_{u}\alpha-\partial_{v}\beta=0,\\
			\partial_{v}\alpha+\partial_{u}\beta=0.
		\end{cases}
	\end{equation*}
	It follows from Proposition \ref{prop-norm-alpha beta} and Proposition \ref{prop-RF} that 
	\begin{equation*}
		\abs{\Re\LR{h_{A}\lr{u+Iv}}}\leq\abs{\alpha\lr{u,v}}\leq\frac{1}{2}\Lr{\abs{f\lr{u-Iv}}+\abs{f\lr{u+Iv}}}
	\end{equation*}
	and
	\begin{equation*}
		\abs{\Im\LR{h_{A}\lr{u+Iv}}}\leq\abs{\beta\lr{u,v}}\leq\frac{1}{2}\Lr{\abs{f\lr{u-Iv}}+\abs{f\lr{u+Iv}}}.
	\end{equation*}
	Hence, we have
	\begin{equation}\label{norm-hA}
		\abs{h_{A}\lr{u+Iv}}\leq\frac{\sqrt{2}}{2}\Lr{\abs{f\lr{u-Iv}}+\abs{f\lr{u+Iv}}}.
	\end{equation}
	\begin{proposition}[\cite{colombo2010extension}, Extension Lemma]\label{prop-ext}
		Let $J\in\mathbb{S}$ and let $D$ be a domain in $L_{J}$, symmetric with respect to the real axis and such that $D\cap\mathbb{R}=\varnothing$. Let $U_{D}$ be the axially symmetric s-domain defined by
		\[U_{D}=\bigcup_{u+Jv\in D,I\in\mathbb{S}}\left(u+Iv\right).\]
		If $f:D\to L_{J}$ is holomorphic, then the function $\mathbf{ext}\left(f\right):U_{D}\to\mathbb{R}_n$ defined by
		\begin{align}\label{formula-ext(f)}
			\begin{split}
				&\mathbf{ext}\left(f\right)\left(u+Iv\right)\\
				=&\frac{1}{2}\left[f\left(u-Jv\right)+f\left(u+Jv\right)\right]+\frac{IJ}{2}\left[f\left(u-Jv\right)-f\left(u+Jv\right)\right]
			\end{split}
		\end{align}
		is the unique left slice monogenic extension of $f$ to $U_{D}$.
		Similarly, if $f:D\to\mathbb{R}_n$ satisfies $\overline{\partial_{J}}f\left(u+Jv\right)=0$, then the function $\ext\lr{f}$ defined by \eqref{formula-ext(f)} is its unique left slice monogenic extension to $U_D$.
	\end{proposition}
	
	\begin{proposition}[\cite{10.1063/1.3636718}]\label{prop-NE}
		Let $U\subseteq\mathbb{R}^{n+1}$ be an axially symmetric s-domain. Let $I_1,I_2,\dots,I_n$ be a completion to a basis of $\mathbb{R}_n$. Then 
		\begin{equation*}
			\mathcal{SM}_{l}\left(U\right)=\oplus_{A}\mathcal{N}_{l}\left(U\right)I_{A}
		\end{equation*}
		where $A\in\mathcal{P}_{<}\lr{\LR{1,\dots,n}}$,
		\begin{equation*}
			I_{A}=
			\begin{cases}
				1,  &A=\varnothing,\\
				I_{i_1}I_{i_2}\dots I_{i_s}, &A=\LR{i_1,\dots,i_s},
			\end{cases}
		\end{equation*}
		and
		\begin{equation*}
			\mathcal{N}_{l}\left(U\right)=\left\{f\in\mathcal{SM}_{l}\left(U\right), f\left(U\cap L_{I}\right)\subseteq L_{I}, I\in\mathbb{S}\right\}.
		\end{equation*}
	\end{proposition}
	%Now, we give the Plancherel theorem for the 1DLCFT as follows.
	
	%\begin{proof}
	%	Let $I=I_1\in\mathbb{S}$. By Proposition \ref{prop-I}, we can choose $I_2,\dots,I_n\in\mathbb{S}$ such that $I=I_{1},I_{2},\dots,I_{n}$ form an orthonormal basis of $\mathbb{R}_{n}$. Using Proposition \ref{prop-SL}, we have
	%    \begin{equation*}
		%        \mathscr{F}_I(f|_{\mathbb{R}})(w)=\sum_{A'}\mathscr{F}_I(F_{A'}|_{\mathbb{R}})(w)I_{A'}.
		%    \end{equation*}
	%	where $F_{A'}$ is $L_{I}$-valued for each $A'\in\mathcal{P}_{<}\lr{\LR{2,\dots,n}}$. Hence, for each $A'\in\mathcal{P}_{<}\lr{\LR{2,\dots,n}}$, we have that $\mathscr{F}_I(F_{A'}|_{\mathbb{R}})$ is $L_{I}$-valued and $$\lVert\mathscr{F}_I(F_{A'}|_{\mathbb{R}})\rVert_{L^{2}(\mathbb{R})}=\lVert F_{A'}|_{\mathbb{R}}\rVert_{L^{2}(\mathbb{R})}.$$ Since $I=I_{1},I_{2},\dots,I_{n}$ form an orthonormal basis of $\mathbb{R}_{n}$, we have
	%	\begin{equation*}
		%		\lVert \mathscr{F}_I(f|_{\mathbb{R}})\rVert_{L^{2}(\mathbb{R})}^{2}=\sum_{A'}\lVert \mathscr{F}_I(F_{A'}|_{\mathbb{R}})\rVert_{L^{2}(\mathbb{R})}^{2}=\sum_{A'}\lVert F_{A'}|_{\mathbb{R}}\rVert_{L^{2}(\mathbb{R})}^{2}=\lVert f|_{\mathbb{R}}\rVert_{L^{2}(\mathbb{R})}^{2}.
		%	\end{equation*}
	%\end{proof}
	\section{Main results}
	Our main results are as follows.
	
	\subsection{Non-compact type Paley-Wiener theorem}
	\ 
	\newline
	\indent
	We recall that the Hardy $H^{2}$ space on the right-half complex plane $\Pi_{+}$ is defined as
	\begin{equation*}
		H^{2}(\Pi_{+})=\LR{f\ \text{is holomorphic on}\ \Pi_{+}:\sup_{x>0}\int_{\mathbb{R}}\abs{f\lr{x+\mathbf{i}y}}^{2}\mathrm{d}y<\infty}.
	\end{equation*}
	It is known that a function $f\in H^{2}(\Pi_{+})$ has the NTBL $f\lr{\mathbf{i}y}$ almost everywhere at $\mathbf{i}y$ and $f\lr{\mathbf{i}y}\in L^{2}(\mathbb{R})$. Moreover, 
	\begin{equation*}
		\sup_{x>0}\int_{\mathbb{R}}\abs{f\lr{x+\mathbf{i}y}}^{2}\mathrm{d}y=\int_{\mathbb{R}}\abs{f\lr{\mathbf{i}y}}^{2}\mathrm{d}y.
	\end{equation*}
	Similarly, the open right half-space of $\mathbb{R}^{n+1}$ is defined as
	$$\mathbb{R}^{n+1}_{+}=\left\{\mathbf{x}=u+Iv:u,v\in\mathbb{R},u>0,I\in\mathbb{S}\right\}.$$
	We set $\Pi_{+,I}=\mathbb{R}_{+}^{n+1}\cap L_{I}$. By Proposition \ref{prop-SL}, we have 
	$$f_{I}\lr{u+Iv}=\sum_{A'}F_{A'}\lr{u+Iv}I_{A'}$$
	where $A'\in\mathcal{P}_{<}\lr{\LR{2,\dots,n}}$ and
	\begin{equation*}
		I_{A'}=
		\begin{cases}
			1,  &A'=\varnothing,\\
			I_{i_1}I_{i_2}\dots I_{i_s}, &A'=\LR{i_1,\dots,i_s}.
		\end{cases}
	\end{equation*}
	Since all the NTBL of $F_{A'}$ exist almost everywhere at $Iv$, also the NTBL of $f$ exists almost everywhere at $Iv$ on $\Pi_{+,I}$ and $f\lr{Iv}=\sum_{A'}F_{A'}\lr{Iv}I_{A'}$ almost everywhere $Iv$. It follows that
	\begin{align*}
		\sup_{u>0}\int_{\mathbb{R}}\abs{f\lr{u+Iv}}^{2}\mathrm{d}v=&\sup_{u>0}\int_{\mathbb{R}}\sum_{A'}\abs{F_{A'}\lr{u+Iv}}^{2}\mathrm{d}v\\
		\leq&\sum_{A'}\sup_{u>0}\int_{\mathbb{R}}\abs{F_{A'}\lr{u+Iv}}^{2}\mathrm{d}v\\
		=&\sum_{A'}\int_{\mathbb{R}}\abs{F_{A'}\lr{Iv}}^{2}\mathrm{d}v\\
		=&\int_{\mathbb{R}}\abs{f\lr{Iv}}^{2}\mathrm{d}v.
	\end{align*}
	So, we define
	\begin{equation*}
		H_{\text{slice}}^{2}\left(\Pi_{+,I}\right)=\left\{f\in\mathcal{SM}_{l}\left(\mathbb{R}^{n+1}_{+}\right);\int_{\mathbb{R}}\lvert f\left(Iv\right)\rvert^{2}\mathrm{d}v<\infty\right\},
	\end{equation*}
	where $f\left(Iv\right)$ denotes the NTBL of $f$ at $Iv$ on $\Pi_{+,I}$. 
	
	In $H_{\text{slice}}^{2}\left(\Pi_{+,I}\right)$ we define the inner product
	\begin{equation*}
		\innerproduct{f}{g}_{H_{\text{slice}}^{2}\left(\Pi_{+,I}\right)}=\int_{\mathbb{R}}\overline{g_{I}\lr{Iv}}f_{I}\lr{Iv}\mathrm{d}v.
	\end{equation*}
	where $f_{I}\left(Iv\right)$, $g_{I}\left(Iv\right)$ denote the NTBL of $f,g$ at $Iv$ on $\Pi_{+,I}$. This inner product gives the norm
	\begin{equation*}
		\norm{f}_{H_{\text{slice}}^{2}\left(\Pi_{+,I}\right)}=\lr{\int_{\mathbb{R}}\abs{f_{I}\lr{Iv}}^{2}\mathrm{d}v}^{\frac{1}{2}}.
	\end{equation*}
	\begin{proposition}
		A function $f\in H_{\text{slice}}^{2}\left(\Pi_{+,I}\right)$ for some $I\in\mathbb{S}$ if and only if $f\in H_{\text{slice}}^{2}\left(\Pi_{+,J}\right)$ for all $J\in\mathbb{S}$.
	\end{proposition}
	\begin{proof}
		By Proposition \ref{prop-RF}, we have
		\[\abs{f\lr{u+Jv}}\leq\frac{\sqrt{2}}{2}\lr{\abs{f\lr{u+Iv}}+\abs{f\lr{u-Iv}}}\]
		and
		\[\abs{f\lr{u+Jv}}^{2}\leq\frac{1}{2}\Lr{\abs{f\lr{u+Iv}}^{2}+\abs{f\lr{u-Iv}}^{2}}.\]
		It follows that
		\begin{align*}
			\norm{f}_{H_{\text{slice}}^{2}\left(\Pi_{+,J}\right)}^{2}=&\int_{\mathbb{R}}\abs{f_{J}\lr{Jv}}^{2}\mathrm{d}v\\
			=&\sup_{u>0}\int_{\mathbb{R}}\abs{f_{J}\lr{u+Jv}}^{2}\mathrm{d}v\\
			\leq&\sup_{u>0}\int_{\mathbb{R}}\frac{1}{2}\Lr{\abs{f_{I}\lr{u+Iv}}^{2}+\abs{f_{I}\lr{u-Iv}}^{2}}\mathrm{d}v\\
			\leq&\frac{1}{2}\Lr{\sup_{u>0}\int_{\mathbb{R}}\abs{f_{I}\lr{u+Iv}}^{2}\mathrm{d}v+\sup_{u>0}\int_{\mathbb{R}}\abs{f_{I}\lr{u-Iv}}^{2}\mathrm{d}v}\\
			=&\int_{\mathbb{R}}\abs{f_{I}\lr{Iv}}^{2}\mathrm{d}v\\
			=&\norm{f}_{H_{\text{slice}}^{2}\left(\Pi_{+,I}\right)}^{2}.
		\end{align*}
		By changing $J$ with $I$ we obtain the reverse inequality. 
	\end{proof}
	We now introduce the left slice monogenic Hardy space of the right-half space $\mathbb{R}_{+}^{n+1}.$
	
	\begin{definition}
		For $1\leq p\leq2$. We define $H_{\text{slice}}^{p}\left(\mathbb{R}_{+}^{n+1}\right)$ as the space of functions $f\in\mathcal{SM}_l\left(\mathbb{R}_{+}^{n+1}\right)$ such that
		\[\sup_{I\in\mathbb{S}}\int_{\mathbb{R}}|f\left(Iv\right)|^{p}\mathrm{d}v<\infty.\]
	\end{definition}
	It is easy to see that $f\in H_{\text{slice}}^{p}\left(\mathbb{R}_{+}^{n+1}\right)$ is equivalent to $f^{\frac{p}{2}}\in H_{\text{slice}}^{2}\left(\mathbb{R}_{+}^{n+1}\right)$.
	\begin{theorem}\label{PW-Rn+1Rn-Hp-FT}
		Let $f\in H_{\text{slice}}^{p}\left(\mathbb{R}_{+}^{n+1}\right)$, $1\leq p\leq2$. Then for every $J\in\mathbb{S}$, we have $\supp\mathscr{F}_{J}\lr{f|_{J\mathbb{R}}}\subset\left(-\infty,0\right]$ and
		\begin{equation*}
			f\lr{u+Jv}=\frac{1}{\sqrt{2\pi}}\int_{-\infty}^{0}e^{\lr{u+Jv}w}\mathscr{F}_{J}\lr{f|_{J\mathbb{R}}}\lr{w}\mathrm{d}w    
		\end{equation*}
		where $f|_{J\mathbb{R}}$ is the NTBL of $f$.
	\end{theorem}
	\begin{theorem}\label{PW-Rn+1Rn-FT-Hp}
		Assume that there exists a purely imaginary unit paravector $I\in\mathbb{S}$, such that $f|_{I\mathbb{R}}\in L^{p}(\mathbb{R})$, $1\leq p\leq2$ and $\supp\FT{I}{f|_{I\mathbb{R}}}\subset\left(-\infty,0\right]$. Then $f\in H_{\text{slice}}^{p}\lr{\mathbb{R}_{+}^{n+1}}$.
	\end{theorem}
	
	In order to prove Theorem \ref{PW-Rn+1Rn-Hp-FT}, we need the following results.
	\begin{lemma}\label{PW-PILI-H2-FT}
		For any fixed $I\in\mathbb{S}$, let $f:\Pi_{+,I}\to L_{I}$ and $f\in H_{\text{slice}}^{2}\lr{\Pi_{+,I}}$. Then there exists a function $G:\mathbb{R}\to L_{I}$ such that $\supp G\subset\left(-\infty,0\right]$, $G\in L^2(-\infty,0)$ and 
		\begin{equation*}
			f\lr{z}=\frac{1}{\sqrt{2\pi}}\int_{-\infty}^{0}e^{zw}G\lr{w}\mathrm{d}w,z\in\Pi_{+,I}.
		\end{equation*}
		Moreover, $G$ is the 1DLCFT of $f|_{I\mathbb{R}}$
		where $f|_{I\mathbb{R}}$ is the NTBL of $f$.
	\end{lemma}
	Based on the argument in  \cite[Theorem 19.2]{rudin1987real}, we immediately obtain Lemma \ref{PW-PILI-H2-FT}. For completeness, we include a brief proof.
	\begin{proof}
		For any given $\epsilon_{1},\epsilon_{2}\in\lr{0,\infty}$ with $\epsilon_{1}>\epsilon_{2}$. For each $\delta>0$, let $\Box_{\delta}$ be the counterclockwise rectangular path with vertices $\epsilon_{1}\pm I\delta$ and $\epsilon_{2}\pm I\delta$. Cauchy's theorem implies that 
		\begin{equation}\label{PW-PILI-H2-FT-Cef}
			\int_{\Box_{\delta}}e^{-wz}f\lr{z}\mathrm{d}z=0.
		\end{equation}
		In the following, we assume $w\in\mathbb{R}$. Let
		\[\Psi\lr{\delta}=\int_{\epsilon_{1}+I\delta}^{\epsilon_{2}+I\delta}e^{-wz}f(z)\mathrm{d}z=\int_{\epsilon_{1}}^{\epsilon_{2}}e^{-w\lr{u+I\delta}}f\lr{u+I\delta}\mathrm{d}u.\]
		By the Cauchy-Schwarz inequality, we get
		\begin{align}\label{PW-PILI-H2-FT-abs(Psi)2}
			\begin{split}
				\abs{\Psi\lr{\delta}}^2&=\Big|\int_{\epsilon_{1}}^{\epsilon_{2}}e^{-w\lr{u+I\delta}}f\lr{u+I\delta}\mathrm{d}u\Big|^2\\
				&\leq\int_{\epsilon_{1}}^{\epsilon_{2}}e^{-2wu}\mathrm{d}u\int_{\epsilon_{1}}^{\epsilon_{2}}\abs{f\lr{u+I\delta}}^2\mathrm{d}u.
			\end{split}
		\end{align}
		Put
		\[\Phi\lr{\delta}=\int_{\epsilon_{1}}^{\epsilon_{2}}\abs{f\lr{u+I\delta}}^2\mathrm{d}u.\]
		Using $f\in H_{\text{slice}}^2\lr{\Pi_{+,I}}$ and Fubini's theorem, we obtain
		\[\int_{\mathbb{R}}\Phi\lr{\delta}\mathrm{d}\delta\leq \text{Constant}.\]
		Thus, there exists a sequence $\LR{\delta_{n}}$ such that $\delta_{n}\to\infty$ as $n\to\infty$ and 
		\[\lim_{n\to\infty}\Lr{\Phi\lr{-\delta_{n}}+\Phi\lr{\delta_{n}}}=0.\]
		It follows from \eqref{PW-PILI-H2-FT-abs(Psi)2} that 
		\begin{equation}\label{PW-PILI-H2-FT-Psi(infty)}
			\lim_{n\to\infty}\Psi\lr{-\delta_{n}}=0,\quad \lim_{n\to\infty}\Psi\lr{\delta_{n}}=0.
		\end{equation}
		Note that \eqref{PW-PILI-H2-FT-Psi(infty)} holds for every $w$ and the sequence $\LR{\delta_{n}}$ does not depend on $w$.
		
		Notice that
		\begin{align*}
			\int_{\epsilon_{2}-I\delta}^{\epsilon_{2}+I\delta}e^{-wz}f\lr{z}\mathrm{d}z=&Ie^{-w\epsilon_{2}}\int_{-\delta}^{\delta}e^{-Iwv}f\lr{\epsilon_{2}+Iv}\mathrm{d}v
		\end{align*}
		and
		\begin{align*}
			\int_{\epsilon_{1}+I\delta}^{\epsilon_{1}-I\delta}e^{-wz}f\lr{z}\mathrm{d}z=&-Ie^{-w\epsilon_{1}}\int_{-\delta}^{\delta}e^{-Iwv}f\lr{\epsilon_{1}+Iv}\mathrm{d}v.
		\end{align*}
		Let $f_{u}\lr{Iv}=f\lr{u+Iv}$ and let
		\begin{equation*}
			g_{n}\lr{u,w}=\frac{1}{\sqrt{2\pi}}\int_{-\delta_{n}}^{\delta_{n}}e^{-Iwv}f_{u}\lr{Iv}\mathrm{d}v,\quad u>0.
		\end{equation*}
		Then we conclude from \eqref{PW-PILI-H2-FT-Cef} and \eqref{PW-PILI-H2-FT-Psi(infty)} that
		\begin{equation}\label{PW-PILI-H2-FT-gx-ge}
			\lim_{n\to\infty}\Lr{e^{-w\epsilon_{2}}g_{n}\lr{\epsilon_{2},w}-e^{-w\epsilon_{1}}g_{n}\lr{\epsilon_{1},w}}=0,\quad w\in\mathbb{R}.
		\end{equation}
		Using $f_{u}\lr{Iv}\in H_{\text{slice}}^2\lr{\Pi_{+,I}}$ and Plancherel's theorem, we have
		\begin{equation*}
			\lim_{n\to\infty}\int_{\mathbb{R}}\Big|\FT{I}{f_{u}}\lr{w}-g_{n}\lr{u,w}\Big|^2\mathrm{d}w=0.
		\end{equation*}
		By invoking the corresponding result in \cite[Theorem 3.12]{rudin1987real}, there exists a subsequence $\LR{g_{n_{k}}\lr{u,w}}$ such that $\LR{g_{n_{k}}\lr{u,w}}$ converges pointwise to $\FT{I}{f_{u}}\lr{w}$ almost everywhere $w$. Thus, by \eqref{PW-PILI-H2-FT-gx-ge}, we get
		\[e^{-w\epsilon_{2}}\FT{I}{f_{\epsilon_{2}}}\lr{w}=e^{-w\epsilon_{1}}\FT{I}{f_{\epsilon_{1}}}\lr{w}.\]
		holds almost everywhere $w\in\mathbb{R}$. If we define $G\lr{w}:=e^{-w\epsilon_{1}}\FT{I}{f_{\epsilon_{1}}}\lr{w}$, by the arbitrariness of $\epsilon_{1},\epsilon_{2}$, we can obtain that for every $u\in\lr{0,\infty}$, 
		\begin{equation}\label{FTfu}
			G\lr{w}=e^{-wu}\FT{I}{f_{u}}\lr{w}
		\end{equation}
		holds almost everywhere $w\in\mathbb{R}$. Hence, $G\lr{w}$ is independent of $u$.
		
		Using \eqref{FTfu} and Plancherel's theorem, we have
		\begin{equation}\label{PW-PILI-H2-FT-e2wuFTf}
			\int_{\mathbb{R}}e^{2wu}\abs{G\lr{w}}^2\mathrm{d}w=\int_{\mathbb{R}}\abs{\FT{I}{f_{u}}\lr{w}}^2\mathrm{d}w=\int_{\mathbb{R}}\abs{f_{u}\lr{v}}^2\mathrm{d}v<\infty.
		\end{equation}
		If $u\to\infty$, we can conclude from \eqref{PW-PILI-H2-FT-e2wuFTf} that $\FT{I}{f|_{I\mathbb{R}}}\lr{w}=0$ holds almost everywhere $w\in\lr{0,\infty}$.
		And if $u\to0$, \eqref{PW-PILI-H2-FT-e2wuFTf} shows that 
		\begin{equation*}
			\int_{-\infty}^{0}\abs{G\lr{w}}^2\mathrm{d}w<\infty.
		\end{equation*}
		It follows from \eqref{FTfu} that $\FT{I}{f_{u}}\lr{w}\in L^{1}(\mathbb{R})$. Hence, we get
		\begin{align*}
			f\lr{u+Iv}=&f_{u}\lr{Iv}\\
			=&\frac{1}{\sqrt{2\pi}}\int_{\mathbb{R}}e^{Ivw}\FT{I}{f_{u}}\lr{w}\mathrm{d}w\\
			=&\frac{1}{\sqrt{2\pi}}\int_{-\infty}^{0}e^{Ivw}e^{wu}G\lr{w}\mathrm{d}w\\
			=&\frac{1}{\sqrt{2\pi}}\int_{-\infty}^{0}e^{\lr{u+Iv}w}G\lr{w}\mathrm{d}w.
		\end{align*}
		Moreover, we conclude that $G$ is the 1DLCFT of $f|_{I\mathbb{R}}$ where $f|_{\mathbb{R}}$ is the NTBL of $f$. The proof is similar to Lemma \ref{PW-PILI-FT-Hp}.
	\end{proof}
	\begin{lemma}\label{lemma-supabs-Hpnorm}
		For any fixed $I\in\mathbb{S}$, let $f:\Pi_{+,I}\to L_{I}$. Suppose $f\in H_{\text{slice}}^{p}\left(\Pi_{+,I}\right),0<p\leq\infty$, and $\Pi_{+,I,x_0}=\left\{x_0+z:x_0>0,z\in\Pi_{+,I}\right\}\subset\Pi_{+,I}$, then there exists a constant $C=C\left(x_0,p\right)$, depending on $\left(x_0,p\right)$ but not on $f$, such that
		\[\sup_{z\in\Pi_{+,I,x_0}}|f\left(z\right)|\leq C\lVert f\rVert_{H_{\text{slice}}^{p}\left(\Pi_{+,I}\right)}.\]
	\end{lemma}
	\begin{proof}
		The proof is similar to that given in \cite[Proposition 1.3]{bekolle2003lecture}.
	\end{proof}
	\begin{lemma}\label{lemma-fFT(g)-FT(f)g}
		For any fixed $I\in\mathbb{S}$, let $f,\varphi:L_{I}\to L_{I}$. If $f|_{\mathbb{R}}\in L^p(\mathbb{R})$, $\varphi|_{\mathbb{R}}\in L^p(\mathbb{R})\cap L^{1}(\mathbb{R})$, $1\leq p\leq2$, then
		\begin{equation*}
			\int_{\mathbb{R}}f\left(t\right)\mathscr{F}_I\left(\varphi|_{\mathbb{R}}\right)\left(t\right)\mathrm{d}t=\int_{\mathbb{R}}\mathscr{F}_I\left(f|_{\mathbb{R}}\right)\left(t\right)\varphi\left(t\right)\mathrm{d}t.
		\end{equation*}
	\end{lemma}
	\begin{proof}
		The proof is similar to that given in \cite{li2018fourier}. 
	\end{proof}
	\begin{lemma}\label{PW-PILI-Hp-FT}
		For any fixed $I\in\mathbb{S}$, let $f:\Pi_{+,I}\to L_{I}$. If $f\in H_{\text{slice}}^{p}\left(\Pi_{+,I}\right)$, $1\leq p\leq2$, then $\supp\FT{I}{f|_{I\mathbb{R}}}\subset\left(-\infty,0\right]$  where $f|_{I\mathbb{R}}$ is the NTBL of $f$. Moreover, for all $z\in\Pi_{+,I}$, we have
		\begin{align*}
			f\left(z\right)=&\frac{1}{\sqrt{2\pi}}\int_{-\infty}^{0}e^{zw}\mathscr{F}_I\left(f|_{I\mathbb{R}}\right)\left(w\right)\mathrm{d}w.
		\end{align*}
		
	\end{lemma}
	\begin{proof}
		The proof is similar to that given in \cite{li2018fourier}. For any given $\epsilon>0$, let us define
		\[f_{\epsilon}(z)=f\lr{\epsilon+z}\]
		for $z\in\Pi_{+,I}$. By Lemma \ref{lemma-supabs-Hpnorm}, we have that there exists a positive number $C=C(\epsilon,p)$ such that
		\[\sup_{z\in\Pi_{+,I}}\abs{f_{\epsilon}\lr{z}}\leq C\norm{f}_{H_{\text{slice}}^{p}\left(\Pi_{+,I}\right)}<\infty.\]
		It follows that for any $u>0$,
		\begin{align*}
			\int_{\mathbb{R}}\abs{f_{\epsilon}\lr{u+Iv}}^{2}\mathrm{d}v=&\int_{\mathbb{R}}\abs{f_{\epsilon}\lr{u+Iv}}^{p}\abs{f_{\epsilon}\lr{u+Iv}}^{2-p}\mathrm{d}v\\
			\leq&\lr{C\norm{f}_{H_{\text{slice}}^{p}\left(\Pi_{+,I}\right)}}^{2-p}\norm{f_{\epsilon}|_{I\mathbb{R}}}_{L^{p}(\mathbb{R})}^{p}<\infty.
		\end{align*}
		Thus, we have $f_{\epsilon}\lr{z}\in H_{\text{slice}}^{2}\lr{\Pi_{+,I}}$. By Lemma \ref{PW-PILI-H2-FT}, we get 
		\begin{equation}\label{PW-PILI-Hp-FT-suppfdelta}
			\supp\FT{I}{f_{\epsilon}|_{I\mathbb{R}}}\subset\left(-\infty,0\right].
		\end{equation}
		Using Hausdorff-Young's inequality, we obtain that
		\begin{equation}\label{lq<lp}
			\norm{\FT{I}{f_{\epsilon}|_{I\mathbb{R}}}-\FT{I}{f|_{I\mathbb{R}}}}_{L^{q}(\mathbb{R})}\leq\norm{f_{\epsilon}|_{I\mathbb{R}}-f|_{I\mathbb{R}}}_{L^{p}(\mathbb{R})}
		\end{equation}
		holds for $1\leq p\leq2$, where $\frac{1}{p}+\frac{1}{q}=1$.
		Since 
		\[\lim_{\epsilon\to0}\norm{f_{\epsilon}|_{I\mathbb{R}}-f|_{I\mathbb{R}}}_{L^{p}(\mathbb{R})}=0,\]
		we have 
		\[\lim_{\epsilon\to0}\norm{\FT{I}{f_{\epsilon}|_{I\mathbb{R}}}-\FT{I}{f|_{I\mathbb{R}}}}_{L^{q}(\mathbb{R})}=0.\]
		Owing to \eqref{PW-PILI-Hp-FT-suppfdelta} and \eqref{lq<lp}, for any rapidly decreasing functions $\varphi\in\mathscr{S}(\mathbb{R})$ satisfying $\supp\varphi\subset\left[0,\infty\right)$, we have
		\begin{align*}
			&\Abs{\int_{\mathbb{R}}\FT{I}{f|_{I\mathbb{R}}}\lr{w}\varphi\lr{w}\mathrm{d}w}\\
			=&\Abs{\int_{\mathbb{R}}\Lr{\FT{I}{f|_{I\mathbb{R}}}\lr{w}-\FT{I}{f_{\epsilon}|_{I\mathbb{R}}}\lr{w}}\varphi\lr{w}\mathrm{d}w+\int_{\mathbb{R}}\FT{I}{f_{\epsilon}|_{I\mathbb{R}}}\lr{w}\varphi\lr{w}\mathrm{d}w}\\
			=&\Abs{\int_{\mathbb{R}}\Lr{\FT{I}{f|_{I\mathbb{R}}}\lr{w}-\FT{I}{f_{\epsilon}|_{I\mathbb{R}}}\lr{w}}\varphi\lr{w}\mathrm{d}w}\\
			\leq&\norm{\FT{I}{f_{\epsilon}|_{I\mathbb{R}}}-\FT{I}{f|_{I\mathbb{R}}}}_{L^{q}(\mathbb{R})}\norm{\varphi}_{L^{p}(\mathbb{R})}.
		\end{align*}
		It follows that $\supp\FT{I}{f|_{I\mathbb{R}}}\subset\left(-\infty,0\right]$. Let 
		\begin{equation*}
			g(u+Iv)=\frac{1}{\sqrt{2\pi}}\int_{-\infty}^{0}e^{\lr{u+Iv}w}\FT{I}{f|_{I\mathbb{R}}}(w)\mathrm{d}w.
		\end{equation*}
		It is easy to show that $g$ is holomorphic on $\Pi_{+,I}$. Note that for fixed $u_{0}>0$, 
		\begin{align*}
			g(u_{0}+Iv)=&\frac{1}{\sqrt{2\pi}}\int_{-\infty}^{0}e^{\lr{u_{0}+Iv}w}\FT{I}{f|_{I\mathbb{R}}}(w)\mathrm{d}w\\
			=&\frac{1}{\sqrt{2\pi}}\int_{-\infty}^{0}e^{Ivw}e^{u_{0}w}\FT{I}{f|_{I\mathbb{R}}}(w)\mathrm{d}w\\
			=&\frac{1}{\sqrt{2\pi}}\int_{-\infty}^{0}e^{Ivw}\FT{I}{f\lr{u_{0}+Iv}}(w)\mathrm{d}w\\
			=&\frac{1}{\sqrt{2\pi}}\int_{-\infty}^{0}e^{Ivw}\FT{I}{f_{u_{0}}|_{I\mathbb{R}}}(w)\mathrm{d}w\\
			=&f_{u_{0}}\lr{Iv}=f(u_{0}+Iv).
		\end{align*}
		Hence, by $f\in H_{\text{slice}}^{p}\left(\Pi_{+,I}\right)$, we have 
		\begin{equation*}
			f\lr{u+Iv}=g\lr{u+Iv}=\frac{1}{\sqrt{2\pi}}\int_{-\infty}^{0}e^{\lr{u+Iv}w}\FT{I}{f|_{I\mathbb{R}}}(w)\mathrm{d}w.
		\end{equation*}
	\end{proof}
	\begin{lemma}\label{PW-NRn+1+-Hp-FT}
		If there exists a unit purely imaginary paravector $I\in\mathbb{S}$, such that $f\in H_{\text{slice}}^{p}\left(\Pi_{+,I}\right)\cap\mathcal{N}_{l}\left(\mathbb{R}_{+}^{n+1}\right)$, $1\leq p\leq2$. Then for every $J\in\mathbb{S}$, $\mathbf{supp}\mathscr{F}_{J}\left(f|_{J\mathbb{R}}\right)\subset\left(-\infty,0\right]$ and 
		\[f\left(u+Iv\right)=\frac{1}{\sqrt{2\pi}}\int_{-\infty}^{0}e^{\left(u+Iv\right)t}\mathscr{F}_{J}\left(f|_{J\mathbb{R}}\right)\left(t\right)\mathrm{d}t,\]
		where $f|_{J\mathbb{R}}$ is the NTBL of $f$.
	\end{lemma}
	\begin{proof}
		For given $I\in\mathbb{S}$, by $f\in\mathcal{N}_{l}\lr{\mathbb{R}_{+}^{n+1}}$, we get $f\lr{\Pi_{+,I}}\subseteq L_{I}$ and $\overline{f\lr{u+Iv}}=f\lr{u-Iv}$. Since $f\in H_{\text{slice}}^{p}\left(\Pi_{+,I}\right)$, using Lemma \ref{PW-PILI-Hp-FT}, we have $\supp\FT{I}{f|_{I\mathbb{R}}}\subset\left(-\infty,0\right]$ and 
		\begin{equation}\label{PW-NRn+1+-Hp-FT-f}
			f\lr{u+Iv}=\frac{1}{\sqrt{2\pi}}\int_{-\infty}^{0}e^{\lr{u+Iv}w}\FT{I}{f|_{I\mathbb{R}}}\lr{w}\mathrm{d}w.
		\end{equation}
		Taking the conjugation of formula \eqref{PW-NRn+1+-Hp-FT-f}, we obtain
		\begin{align}\label{PW-NRn+1+-Hp-FT-Cf}
			\begin{split}
				\overline{f\lr{u+Iv}}=&\overline{\frac{1}{\sqrt{2\pi}}\int_{-\infty}^{0}e^{\lr{u+Iv}w}\FT{I}{f|_{I\mathbb{R}}}\lr{w}\mathrm{d}w}\\
				=&\frac{1}{\sqrt{2\pi}}\int_{-\infty}^{0}e^{\lr{u-Iv}w}\overline{\FT{I}{f|_{I\mathbb{R}}}\lr{w}}\mathrm{d}w.
			\end{split}
		\end{align}
		Notice that
		\begin{align*}
			\overline{\int_{-T}^{T}e^{-Ivw}f\lr{Iv}\mathrm{d}v}=\int_{-T}^{T}e^{Ivw}\overline{f\lr{Iv}}\mathrm{d}v.
		\end{align*}
		Let $f\lr{Iv}$ be the NTBL of $f\lr{u+Iv}$. Using the fact that $\overline{f\lr{u+Iv}}=f\lr{u-Iv}$, we have that $\overline{f\lr{Iv}}=f\lr{-Iv}$ holds almost everywhere $v\in\mathbb{R}$. It follows that 
		\begin{equation}\label{BarFT=FT}
			\overline{\int_{-T}^{T}e^{-Ivw}f\lr{Iv}\mathrm{d}v}=\int_{-T}^{T}e^{Ivw}f\lr{-Iv}\mathrm{d}v=\int_{-T}^{T}e^{-Ivw}f\lr{Iv}\mathrm{d}v.
		\end{equation}
		Since 
		\[\lim_{T\to\infty}\Norm{\FT{I}{f|_{I\mathbb{R}}}\lr{w}-\frac{1}{\sqrt{2\pi}}\int_{-T}^{T}e^{-Ivw}f\lr{Iv}\mathrm{d}v}_{L^{2}(\mathbb{R})}=0,\]
		we get 
		\begin{equation}\label{PW-NRn+1+-Hp-FT-CFT}
			\overline{\FT{I}{f|_{I\mathbb{R}}}\lr{w}}=\FT{I}{f|_{I\mathbb{R}}}\lr{w}.
		\end{equation}
		Substituting \eqref{PW-NRn+1+-Hp-FT-CFT} into \eqref{PW-NRn+1+-Hp-FT-Cf}, we obtain
		\begin{equation*}
			\overline{f\lr{u+Iv}}=\frac{1}{\sqrt{2\pi}}\int_{-\infty}^{0}e^{\lr{u-Iv}w}\FT{I}{f|_{I\mathbb{R}}}\lr{w}\mathrm{d}w=f\lr{u-Iv}.
		\end{equation*}
		By Proposition \ref{prop-ext}, for every $J\in\mathbb{S}$, we have
		\begin{align*}
			f\lr{u+Jv}:=&\ \ext\lr{f}\lr{u+Jv}\\
			=&\frac{1}{2}\Lr{f\lr{u-Iv}+f\lr{u+Iv}}+\frac{JI}{2}\Lr{f\lr{u-Iv}-f\lr{u+Iv}}\\
			=&\frac{1}{2}\Lr{\overline{f\lr{u+Iv}}+f\lr{u+Iv}}+\frac{JI}{2}\Lr{\overline{f\lr{u+Iv}}-f\lr{u+Iv}}\\
			=&\frac{1}{\sqrt{2\pi}}\int_{-\infty}^{0}\frac{1}{2}\Lr{e^{\lr{u-Iv}w}+e^{\lr{u+Iv}w}}\FT{I}{f|_{I\mathbb{R}}}\lr{w}\mathrm{d}w\\
			&+\frac{1}{\sqrt{2\pi}}\int_{-\infty}^{0}\frac{JI}{2}\Lr{e^{\lr{u-Iv}w}-e^{\lr{u+Iv}w}}\FT{I}{f|_{I\mathbb{R}}}\lr{w}\mathrm{d}w\\
			=&\frac{1}{\sqrt{2\pi}}\int_{-\infty}^{0}\ext\lr{e^{\lr{u+Iv}w}}\lr{u+Jv}\FT{I}{f|_{I\mathbb{R}}}\lr{w}\mathrm{d}w.
		\end{align*}
		It is easy to see that 
		\begin{equation*}
			f\lr{u+Jv}=\Re\LR{f\lr{u+Iv}}+J\Im\LR{f\lr{u+Iv}}.
		\end{equation*}
		Denote by $f\lr{Jv}$ the NTBL of $f\lr{u+Jv}$. It follows that $$f\lr{Jv}=\Re\LR{f\lr{Iv}}+J\Im\LR{f\lr{Iv}}$$ holds almost everywhere $v\in\mathbb{R}$. Note that
		\begin{align*}
			&\int_{-T}^{T}e^{-Jvw}f\lr{Jv}\mathrm{d}v\\
			=&\int_{-T}^{T}\Lr{\cos\lr{vw}-J\sin\lr{vw}}\Lr{\Re\LR{f\lr{Iv}}+J\Im\LR{f\lr{Iv}}}\mathrm{d}v\\
			=&\int_{-T}^{T}\Lr{\cos\lr{vw}\Re\LR{f\lr{Iv}}+\sin\lr{vw}\Im\LR{f\lr{Iv}}}\mathrm{d}v\\
			&+J\int_{-T}^{T}\Lr{\cos\lr{vw}\Im\LR{f\lr{Iv}}-\sin\lr{vw}\Re\LR{f\lr{Iv}}}\mathrm{d}v\\
			=&\Re\LR{\int_{-T}^{T}e^{-Ivw}f\lr{Iv}\mathrm{d}v}+J\Im\LR{\int_{-T}^{T}e^{-Ivw}f\lr{Iv}\mathrm{d}v}.
		\end{align*}
		By \eqref{BarFT=FT}, we have
		\[\Im\LR{\int_{-T}^{T}e^{-Ivw}f\lr{Iv}\mathrm{d}v}=0.\]
		Hence, we get
		\begin{equation*}
			\int_{-T}^{T}e^{-Jvw}f\lr{Jv}\mathrm{d}v=\int_{-T}^{T}e^{-Ivw}f\lr{Iv}\mathrm{d}v.
		\end{equation*}
		By
		\[\lim_{T\to\infty}\Norm{\FT{J}{f|_{J\mathbb{R}}}\lr{w}-\frac{1}{\sqrt{2\pi}}\int_{-T}^{T}e^{-Jvw}f\lr{Jv}\mathrm{d}v}_{L^{2}(\mathbb{R})}=0,\]
		we have $\FT{J}{f|_{J\mathbb{R}}}\lr{w}=\FT{I}{f|_{I\mathbb{R}}}\lr{w}$. Since 
		$\overline{e^{\lr{u+Iv}w}}=e^{\lr{u-Iv}w},$ 
		we get $$\ext\lr{e^{\lr{u-Iv}w}}\lr{u+Jv}=e^{\lr{u+Jv}w}.$$ Hence, we have $\supp\FT{J}{f|_{J\mathbb{R}}}=\supp\FT{I}{f|_{I\mathbb{R}}}\subset\left(-\infty,0\right]$ and
		\begin{equation*}
			f\lr{u+Jv}=\frac{1}{\sqrt{2\pi}}\int_{-\infty}^{0}e^{\lr{u+Jv}w}\FT{J}{f|_{J\mathbb{R}}}\lr{w}\mathrm{d}w.
		\end{equation*}
	\end{proof}
	
	Now, we prove Theorem \ref{PW-Rn+1Rn-Hp-FT}.
	\begin{proof}[Proof of Theorem \ref{PW-Rn+1Rn-Hp-FT}]
		Since $f\in H_{\text{slice}}^{p}\left(\mathbb{R}_{+}^{n+1}\right)$, by Proposition \ref{prop-NE}, for every $J\in\mathbb{S}$, we have
		\begin{equation}\label{PW-Rn+1Rn-Hp-FT-f}
			f\lr{u+Jv}=\sum_{A}h_{A}\lr{u+Jv}I_{A},
		\end{equation}
		where $h_{A}\in\mathcal{N}_{l}\lr{\mathbb{R}_{+}^{n+1}}$ and $A\in\mathcal{P}_{<}\lr{\LR{1,\dots,n}}$. By \eqref{norm-hA}, we have $$\abs{h_{A}\lr{u+Jv}}\leq \frac{\sqrt{2}}{2}\Lr{\abs{f\lr{u-Jv}}+\abs{f\lr{u+Jv}}}.$$
		Hence, we have $h_{A}\in H_{\text{slice}}^{p}\left(\Pi_{+,J}\right)$. It follow that $\supp\FT{J}{h_{A}|_{J\mathbb{R}}}\subset\left(-\infty,0\right]$ and
		\begin{equation}\label{PW-Rn+1Rn-Hp-FT-hA}
			h_{A}\lr{u+Jv}=\frac{1}{\sqrt{2\pi}}\int_{-\infty}^{0}e^{\lr{u+Jv}w}\FT{J}{h_{A}|_{J\mathbb{R}}}\lr{w}\mathrm{d}w.
		\end{equation}
		Combining \eqref{PW-Rn+1Rn-Hp-FT-f} and \eqref{PW-Rn+1Rn-Hp-FT-hA}, we get
		\begin{align*}
			f\lr{u+Jv}=&\sum_{A}\frac{1}{\sqrt{2\pi}}\int_{-\infty}^{0}e^{\lr{u+Jv}w}\FT{J}{h_{A}|_{J\mathbb{R}}}\lr{w}\mathrm{d}wI_{A}\\
			&=\frac{1}{\sqrt{2\pi}}\int_{-\infty}^{0}e^{\lr{u+Jv}w}\sum_{A}\FT{J}{h_{A}|_{J\mathbb{R}}}\lr{w}I_{A}\mathrm{d}w\\
			&=\frac{1}{\sqrt{2\pi}}\int_{-\infty}^{0}e^{\lr{u+Jv}w}\FT{J}{\sum_{A}h_{A}I_{A}|_{J\mathbb{R}}}\lr{w}\mathrm{d}w\\
			&=\frac{1}{\sqrt{2\pi}}\int_{-\infty}^{0}e^{\lr{u+Jv}w}\FT{J}{f|_{J\mathbb{R}}}\lr{w}\mathrm{d}w
		\end{align*}
		and $\supp\FT{J}{f|_{J\mathbb{R}}}=\bigcup_{A}\supp\FT{J}{h_{A}|_{J\mathbb{R}}}\subset\left(-\infty,0\right]$.
	\end{proof}
	
	Next, we need the following result to prove Theorem \ref{PW-Rn+1Rn-FT-Hp}.
	\begin{lemma}\label{PW-PILI-FT-Hp}
		For any fixed $I\in\mathbb{S}$, let $f:\Pi_{+,I}\to L_{I}$ and let $f|_{I\mathbb{R}}$ be the NTBL of $f$. Assume that $f|_{I\mathbb{R}}\in L^{p}(\mathbb{R})$, $1\leq p\leq2$ and $\supp\FT{I}{f|_{I\mathbb{R}}}\subset\left(-\infty,0\right]$. Then $f\in H_{\text{slice}}^{p}\left(\Pi_{+,I}\right)$. Moreover, we have that
		\begin{align}
			\begin{split}
				f\left(z\right)=&\frac{1}{\sqrt{2\pi}}\int_{-\infty}^{0}e^{zw}\mathscr{F}_I\left(f|_{I\mathbb{R}}\right)\left(w\right)\mathrm{d}w\\
				=&\int_{\mathbb{R}}\mathscr{K}\left(z-Iw\right)f(Iw)\mathrm{d}w\\
				=&\int_{\mathbb{R}}\mathscr{P}\left(u,v-w\right)f\left(Iw\right)\mathrm{d}w 
			\end{split}
		\end{align}
		holds for all $z=u+Iv\in\Pi_{+,I}$, where 
		\begin{equation}
			\mathscr{K}\left(z\right)=\frac{1}{2\pi}\int_{-\infty}^{0}e^{zw}\mathrm{d}w,\ \mathscr{P}\left(u,v\right)=\frac{|\mathscr{K}\left(z\right)|^2}{\mathscr{K}\left(2u\right)}.
		\end{equation}
	\end{lemma}
	\begin{proof}
		The proof is similar to that given in \cite{li2018fourier} (see also \cite{Hao-Dang-Mai}). For any fixed $I\in\mathbb{S}$. It is easy to show that  for all $z=u+Iv\in\Pi_{+,I}$, $$e^{zw}\FT{I}{f|_{I\mathbb{R}}}\lr{w}\in L^{1}\lr{-\infty,0}.$$
		Using Fubini's theorem and Morera' theorem, we conclude that
		\begin{equation}
			g\lr{z}=\frac{1}{\sqrt{2\pi}}\int_{-\infty}^{0}e^{zw}\FT{I}{f|_{I\mathbb{R}}}\lr{w}\mathrm{d}w
		\end{equation}\label{PW-PILI-FT-Hp-g}
		is holomorphic on $\Pi_{+,I}$. Next, we will show that $g\in H_{\text{slice}}^{p}\left(\Pi_{+,I}\right)$. For every $\xi,\eta\in \Pi_{+,I}$, define
		\begin{equation*}
			\mathscr{K}\lr{\xi}=\frac{1}{2\pi}\int_{-\infty}^{0}e^{\xi t}\mathrm{d}t
		\end{equation*}
		and 
		\begin{equation}\label{PW-PILI-FT-Hp-Q}
			Q\lr{\xi,\eta}=\int_{\mathbb{R}}\mathscr{K}\lr{\xi-Iw}\mathscr{K}\lr{\eta+Iw}f\lr{Iw}\mathrm{d}w.
		\end{equation}
		Note that 
		\begin{align*}
			\mathscr{K}\lr{\xi-Iw}\mathscr{K}\lr{\eta+Iw}=&\frac{1}{4\pi^2}\int_{-\infty}^{0}e^{\lr{\xi-Iw}t}\mathrm{d}t\int_{-\infty}^{0}e^{\lr{\eta+Iw}s}\mathrm{d}s\\
			=&\frac{1}{4\pi^2}\int_{-\infty}^{0}\int_{-\infty}^{0}e^{\lr{\xi-Iw}t}e^{\lr{\eta+Iw}s}\mathrm{d}t\mathrm{d}s\\
			=&\frac{1}{4\pi^2}\int_{-\infty}^{0}\int_{-\infty}^{0}e^{\xi t}e^{\eta s}e^{-Iwt}e^{Iws}\mathrm{d}t\mathrm{d}s\\
			=&\frac{1}{4\pi^2}\int_{-\infty}^{0}\int_{-\infty}^{0}e^{\xi t}e^{\eta s}e^{-Iw\lr{t-s}}\mathrm{d}t\mathrm{d}s.
		\end{align*}
		Let $r=t-s$. We have
		\begin{align*}
			\mathscr{K}\lr{\xi-Iw}\mathscr{K}\lr{\eta+Iw}=&\frac{1}{4\pi^2}\int_{-\infty}^{0}\int_{-\infty}^{-s}e^{\xi \lr{r+s}}e^{\eta s}e^{-Iwr}\mathrm{d}r\mathrm{d}s\\
			=&\frac{1}{4\pi^2}\int_{-\infty}^{-s}e^{\xi r}e^{-Iwr}\mathrm{d}r\int_{-\infty}^{0}e^{\lr{\xi+\eta}s}\mathrm{d}s\\
			=&\mathscr{K}\lr{\xi+\eta}\frac{1}{2\pi}\int_{-\infty}^{-s}e^{\xi r}e^{-Iwr}\mathrm{d}r
		\end{align*}
		where $s<0$. It follows that 
		\begin{align*}
			Q\lr{\xi,\eta}=&\int_{\mathbb{R}}\mathscr{K}\lr{\xi-Iw}\mathscr{K}\lr{\eta+Iw}f\lr{Iw}\mathrm{d}w\\
			=&\mathscr{K}\lr{\xi+\eta}\frac{1}{2\pi}\int_{\mathbb{R}}\int_{-\infty}^{-s}e^{\xi r}e^{-Iwr}f\lr{Iw}\mathrm{d}r\mathrm{d}w\\
			=&\mathscr{K}\lr{\xi+\eta}\frac{1}{\sqrt{2\pi}}\int_{-\infty}^{-s}e^{\xi r}\FT{I}{f|_{I\mathbb{R}}}\lr{r}\mathrm{d}r\\
			=&\mathscr{K}\lr{\xi+\eta}\frac{1}{\sqrt{2\pi}}\int_{-\infty}^{0}e^{\xi r}\FT{I}{f|_{I\mathbb{R}}}\lr{r}\mathrm{d}r\\
			=&\mathscr{K}\lr{\xi+\eta}g\lr{\xi}.
		\end{align*}
		For $z=u+Iv\in\Pi_{+,I}$, we get $\overline{z}\in\Pi_{+,I}$, and 
		\begin{equation}\label{PW-PILI-FT-Hp-Kg}
			Q\lr{z,\overline{z}}=\mathscr{K}\lr{z+\overline{z}}g\lr{z}=\mathscr{K}\lr{2u}g\lr{z}.
		\end{equation}
		Combining \eqref{PW-PILI-FT-Hp-Q} and \eqref{PW-PILI-FT-Hp-Kg}, we have
		\begin{equation}\label{PW-PILI-FT-Hp-Pf}
			g\lr{z}=\int_{\mathbb{R}}\frac{\mathscr{K}\lr{z-Iw}\mathscr{K}\lr{\overline{z}+Iw}}{\mathscr{K}\lr{2u}}f\lr{Iw}\mathrm{d}w=\int_{\mathbb{R}}\mathscr{P}\lr{u,v-w}f\lr{Iw}\mathrm{d}w
		\end{equation}
		where $\mathscr{P}\lr{u,v}$ is the Poisson kernel. Since $\norm{g}_{L^{p}(\mathbb{R})}=\norm{\mathscr{P}*f|_{I\mathbb{R}}}_{L^{p}(\mathbb{R})}\leq\norm{\mathscr{P}}_{L^{1}(\mathbb{R})}\norm{f|_{I\mathbb{R}}}_{L^{p}(\mathbb{R})}$, we have 
		\begin{equation*}
			\sup_{u>0}\int_{\mathbb{R}}\abs{g\lr{u+Iv}}^{p}\mathrm{d}v\leq\int_{\mathbb{R}}\abs{f\lr{Iv}}^{p}\mathrm{d}v<\infty.
		\end{equation*}
		It follows that $g\in H_{\text{slice}}^{p}\left(\Pi_{+,I}\right)$. By \eqref{PW-PILI-FT-Hp-Pf}, we obtain that $\lim_{u\to0}g\lr{u+Iv}=f\lr{Iv}$ holds almost everywhere $v\in\mathbb{R}$. Assume that $f\lr{Iv}$ is the NTBL of a holomorphic function $f\lr{u+Iv}$ at $Iv$. We can easily show that $f\lr{z}=g\lr{z}$ holds for all $z\in\Pi_{+,I}$. Hence, we have $f\in H_{\text{slice}}^{p}\left(\Pi_{+,I}\right)$.
		
		Moreover, by Lemma \ref{lemma-fFT(g)-FT(f)g} and \eqref{PW-PILI-FT-Hp-g}, we have
		\begin{align*}
			f\lr{z}=&\frac{1}{\sqrt{2\pi}}\int_{-\infty}^{0}e^{zw}\FT{I}{f|_{I\mathbb{R}}}\lr{w}\mathrm{d}w\\
			=&\frac{1}{2\pi}\int_{\mathbb{R}}\int_{-\infty}^{0}e^{-Ivw}e^{zw}\mathrm{d}wf\lr{Iv}\mathrm{d}v\\
			=&\frac{1}{2\pi}\int_{\mathbb{R}}\mathscr{K}\lr{z-Iv}f\lr{Iv}\mathrm{d}v.
		\end{align*}
	\end{proof}
	
	Now, we prove Theorem \ref{PW-Rn+1Rn-FT-Hp}.
	\begin{proof}[Proof of Theorem \ref{PW-Rn+1Rn-FT-Hp}]
		For given $I=I_{1}\in\mathbb{S}$, let $f:I\mathbb{R}\to\mathbb{R}_{n}$ and let $I_{2},\dots,I_{n}$ be
		a completion to a basis of  $\mathbb{R}_{n}$ satisfying  $I_{r}I_{s}+I_{s}I_{r}=-2\delta_{rs}$. By Proposition \ref{prop-SL}, we have 
		\[f(Iv)=\sum_{A^{'}}F_{A^{'}}(Iv)I_{A^{'}}\]
		where $A^{'}\in\mathcal{P}_{<}\lr{\LR{2,\dots,n}}$ and $F_{A^{'}}$ is $L_{I}$-valued. It is easy to verify that
		\[|f(Iv)|^{2}=\sum_{A^{'}}|F_{A^{'}}(Iv)|^{2}\]
		and
		\[|\FT{I}{f|_{I\mathbb{R}}}|^{2}=\sum_{A^{'}}|\FT{I}{F_{A^{'}}|_{I\mathbb{R}}}|^{2}.\]
		Since $f|_{I\mathbb{R}}\in L^{p}(\mathbb{R})$, $1\leq p\leq2$ and $\supp\FT{I}{f|_{I\mathbb{R}}}\subset\left(-\infty,0\right]$, for each $A^{'}\in\mathcal{P}_{<}\lr{\LR{2,\dots,n}}$, we have $F_{A^{'}}|_{I\mathbb{R}}\in L^{p}(\mathbb{R})$ and $\supp\FT{I}{F_{A^{'}}|_{I\mathbb{R}}}\subset(-\infty,0]$. By Lemma \ref{PW-PILI-FT-Hp}, we obtain that for each $A^{'}\in\mathcal{P}_{<}\lr{\LR{2,\dots,n}}$,
		\begin{equation*}
			F_{A^{'}}\lr{u+Iv}=\int_{\mathbb{R}}\mathscr{P}\lr{u,v-w}F_{A^{'}}\lr{Iw}\mathrm{d}w
		\end{equation*}
		is in $H_{\text{slice}}^{p}\lr{\Pi_{+,I}}$ and  $F_{A^{'}}\lr{Iv}$ is the NTBL of $F_{A^{'}}$ at $Iv$. Let
		\[f\lr{u+Iv}=\sum_{A^{'}}F_{A^{'}}\lr{u+Iv}I_{A^{'}}.\]
		Then
		\begin{equation*}
			f\lr{u+Iv}=\int_{\mathbb{R}}\mathscr{P}\lr{u,v-w}f\lr{Iw}\mathrm{d}w
		\end{equation*}
		is in $H_{\text{slice}}^{p}\lr{\Pi_{+,I}}$ and $f\lr{Iv}$ is the NTBL of $f$ at $Iv$.
		Note that
		\begin{align*}
			f\lr{u-Iv}=&\int_{\mathbb{R}}\mathscr{P}\lr{u,-v-w}f\lr{Iw}\mathrm{d}w\\
			=&\int_{\mathbb{R}}\mathscr{P}\lr{u,-\lr{v+w}}f\lr{Iw}\mathrm{d}w\\
			=&\int_{\mathbb{R}}\mathscr{P}\lr{u,v+w}f\lr{Iw}\mathrm{d}w\\
			=&\int_{\mathbb{R}}\mathscr{P}\lr{u,v-w}f\lr{-Iw}\mathrm{d}w.
		\end{align*}
		By Proposition \ref{prop-ext}, for every $J\in\mathbb{S}$, we have
		\begin{align*}
			f\lr{u+Jv}:=&\ \ext\lr{f}\lr{u+Jv}\\
			=&\frac{1}{2}\Lr{f\lr{u-Iv}+f\lr{u+Iv}}+\frac{JI}{2}\Lr{f\lr{u-Iv}-f\lr{u+Iv}}\\
			=&\int_{\mathbb{R}}\mathscr{P}\lr{u,v-w}\lr{\frac{1}{2}\Lr{f\lr{-Iw}+f\lr{Iw}}+\frac{JI}{2}\Lr{f\lr{-Iw}-f\lr{Iw}}}\mathrm{d}w.
		\end{align*}
		It is obvious that $f$ is left slice monogenic on $\mathbb{R}_{+}^{n+1}$. Denote by $f\lr{Jv}$ the NTBL of $f$ at $Jv$. We obtain that
		\begin{equation*}
			f\lr{Jv}=\frac{1}{2}\Lr{f\lr{-Iv}+f\lr{Iv}}+\frac{JI}{2}\Lr{f\lr{-Iv}-f\lr{Iv}}
		\end{equation*}
		holds almost everywhere $v\in\mathbb{R}$. Note that
		\begin{equation*}
			\int_{\mathbb{R}}\abs{f\lr{Jv}}^{p}\mathrm{d}v\leq2^{-\frac{p}{2}}\int_{\mathbb{R}}\Abs{\abs{f\lr{-Iv}}+\abs{f\lr{Iv}}}^{p}\mathrm{d}v\leq2^{1-\frac{p}{2}}\int_{\mathbb{R}}\abs{f\lr{Iv}}^{p}\mathrm{d}v<\infty.
		\end{equation*}
		Hence, we have $$\sup_{J\in\mathbb{S}}\int_{\mathbb{R}}\abs{f\lr{Jv}}^{p}\mathrm{d}v<\infty.$$
		It follows that $f\in H_{\text{slice}}^{p}\left(\mathbb{R}_{+}^{n+1}\right)$. 
	\end{proof}
	
	Furthermore, we can compute the reproducing kernel for the left slice monogenic Hardy space in the right half-space as an application of Theorem \ref{PW-Rn+1Rn-Hp-FT}.  Define
	\begin{align}
		\begin{split}
			\mathscr{K}_{\Pi_{+,I},\xi+I\eta}\lr{u+Iv}&=\mathscr{K}_{\Pi_{+,I},\xi+I\eta}\lr{u+Iv,\xi+I\eta}\\
			&=\frac{1}{2\pi}\int_{-\infty}^{0}e^{\Lr{\lr{u+Iv}+\lr{\xi-I\eta}}w}\mathrm{d}w.
		\end{split}
	\end{align}
	By Proposition \ref{prop-ext}, for every $J\in\mathbb{S},$ we define
	\begin{align}
		\begin{split}
			&\mathscr{K}_{\mathbb{R}_{+}^{n+1},\xi+I\eta}\lr{u+Jv}\\
			=&\ \ext\lr{\mathscr{K}_{\Pi_{+,I}}\lr{u+Iv,\xi+I\eta}}\lr{u+Jv}\\
			=&\frac{1}{2}\Lr{\frac{1}{2\pi}\int_{-\infty}^{0}e^{\Lr{\lr{u-Iv}+\lr{\xi-I\eta}}w}\mathrm{d}w+\frac{1}{2\pi}\int_{-\infty}^{0}e^{\Lr{\lr{u+Iv}+\lr{\xi-I\eta}}w}\mathrm{d}w}\\
			&+\frac{JI}{2}\Lr{\frac{1}{2\pi}\int_{-\infty}^{0}e^{\Lr{\lr{u-Iv}+\lr{\xi-I\eta}}w}\mathrm{d}w-\frac{1}{2\pi}\int_{-\infty}^{0}e^{\Lr{\lr{u+Iv}+\lr{\xi-I\eta}}w}\mathrm{d}w}\\
			=&\frac{1}{2\pi}\int_{-\infty}^{0}\frac{1}{2}\Lr{e^{\lr{u-Iv}w}+e^{\lr{u+Iv}w}}e^{\lr{\xi-I\eta}w}\mathrm{d}w\\
			&+\frac{1}{2\pi}\int_{-\infty}^{0}\frac{JI}{2}\Lr{e^{\lr{u-Iv}w}+e^{\lr{u+Iv}w}}e^{\lr{\xi-I\eta}w}\mathrm{d}w\\
			=&\frac{1}{2\pi}\int_{-\infty}^{0}\Lr{\Re\LR{e^{\lr{u+Iv}w}}+J\Im\LR{e^{\lr{u+Iv}w}}}e^{\lr{\xi-I\eta}w}\mathrm{d}w\\
			=&\frac{1}{2\pi}\int_{-\infty}^{0}e^{\lr{u+Jv}w}e^{\lr{\xi-I\eta}w}\mathrm{d}w\\
			=&\frac{1}{2\pi}\int_{-\infty}^{0}e^{\Lr{\lr{u+Jv}+\lr{\xi-I\eta}}w}\mathrm{d}w.
		\end{split}
	\end{align}
	
	Now, we give the reproducing kernel formula for $H_{slice}^{p}\left(\mathbb{R}_{+}^{n+1}\right)$.
	\begin{corollary}
		Let $f\in H_{slice}^{p}\left(\mathbb{R}_{+}^{n+1}\right)$ and $I\in\mathbb{S}$. Then for every $\mathbf{x}\in\mathbb{R}^{n+1}$, we have
		\begin{equation*}
			f\lr{\mathbf{x}}=\int_{\mathbb{R}}\mathscr{K}_{\mathbb{R}_{+}^{n+1}}\lr{\mathbf{x},I\xi}f\lr{I\xi}\mathrm{d}\xi,
		\end{equation*}
		where $f\lr{I\xi}$ is the NTBL of $f$ at $I\xi$.
	\end{corollary}
	\begin{proof}
		Using Theorem \ref{PW-Rn+1Rn-Hp-FT} and Fubini's theorem, for given $I\in\mathbb{S}$, we have
		\begin{align*}
			f\lr{u+Iv}=&\frac{1}{\sqrt{2\pi}}\int_{-\infty}^{0}e^{\lr{u+Iv}w}\FT{I}{f|_{I\mathbb{R}}}\lr{w}\mathrm{d}w\\
			=&\frac{1}{2\pi}\int_{\mathbb{R}}\int_{-\infty}^{0}e^{\lr{u+Iv}w}e^{-I\xi w}\mathrm{d}wf\lr{I\xi}\mathrm{d}\xi\\
			=&\int_{\mathbb{R}}\mathscr{K}_{\Pi_{+,I}}\lr{u+Iv,I\xi}f\lr{I\xi}\mathrm{d}\xi.
		\end{align*}
		By Proposition \ref{prop-ext}, for every $J\in\mathbb{S}$, we have
		\[f(x+Jv)=\int_{\mathbb{R}}\mathscr{K}_{\mathbb{R}_{+}^{n+1}}\lr{u+Jv,I\xi}f\lr{I\xi}\mathrm{d}\xi.\]
	\end{proof}
	
	\subsection{Compact type Paley-Wiener theorem}
	
	In this subsection we will prove the following main result.
	\begin{theorem}\label{PW-Rn+1Rn}
		Let $f:\mathbb{R}^{n+1}\to\mathbb{R}_{n}$ be a left slice monogenic function with $f|_{\mathbb{R}}\in L^{2}(\mathbb{R})$ and $B>0$ be a positive number. Then the following two conditions are equivalent:
		\begin{enumerate}
			\item[(i)] There exists a constant $C$ such that 
			\begin{equation*}
				|f\left(u+Iv\right)|\leq Ce^{B|u+Iv|}
			\end{equation*}
			for all $I\in\mathbb{S}$ and $u,v\in\mathbb{R}$.
			\item[(ii)] $\mathbf{supp}\mathscr{F}_{I}(f|_{\mathbb{R}})\subset\Lr{-B,B}$ for all  $I\in\mathbb{S}$.
		\end{enumerate}
		Moreover, if one of the above conditions holds, then for all $I\in\mathbb{S}$ and $u,v\in\mathbb{R}$, we have 
		\begin{equation*}
			f\left(u+Iv\right)=\frac{1}{\sqrt{2\pi}}\int_{-B}^{B}e^{I\left(u+Iv\right)w}\mathscr{F}_{I}\left(f|_{\mathbb{R}}\right)\left(w\right)\mathrm{d}w.
		\end{equation*}
	\end{theorem}
	
	In order to prove Theorem \ref{PW-Rn+1Rn}, we need the following results.
	\begin{lemma}\label{PW-LILI}
		For any fixed $I\in\mathbb{S}$, let $f:L_{I}\to L_{I}$ be a holomorphic function with $f|_{\mathbb{R}}\in L^{2}(\mathbb{R})$ and $B>0$ be a positive number. Then the following two conditions are equivalent:
		\begin{enumerate}
			\item[(i)] There exists a constant $C$ such that 
			\begin{equation}\label{PW-LILI-norm-f}
				|f\left(z\right)|\leq Ce^{B|z|}
			\end{equation}
			for all $z=u+Iv\in L_{I}$.
			\item[(ii)] $\mathbf{supp}\mathscr{F}_{I}(f|_{\mathbb{R}})\subset\Lr{-B,B}$.
		\end{enumerate}
		Moreover, if one of the above conditions holds, then we have 
		\begin{equation}\label{PW-LILI-HFT}
			f\left(z\right)=\frac{1}{\sqrt{2\pi}}\int_{-B}^{B}e^{Izw}\mathscr{F}_{I}\left(f|_{\mathbb{R}}\right)\left(w\right)\mathrm{d}w.
		\end{equation}
	\end{lemma}
	\begin{proof}
		Assume that (i) holds. We will mainly adopt the technique used in \cite[Theorem 19.3]{rudin1987real}. For $\epsilon>0$ and $u\in\mathbb{R}$, let $f_{\epsilon}\left(u\right)=e^{-\epsilon|u|}f\left(u\right)$. Using $f|_{\mathbb{R}}\in L^{2}(\mathbb{R})$ and the Cauchy-Schwarz inequality, we have $f_{\epsilon}\left(u\right)\in L^{1}(\mathbb{R})$. For given $I\in\mathbb{S}$, the Fourier transform of $f_{\epsilon}$ is 
		\begin{equation*}
			\mathscr{F}_{I}\left(f_{\epsilon}\right)\left(w\right)=\frac{1}{\sqrt{2\pi}}\int_{\mathbb{R}}e^{-Iuw}f_{\epsilon}\left(u\right)\mathrm{d}u.
		\end{equation*}
		Since $$\lim_{\epsilon\to0}\lVert f_{\epsilon}-f|_{\mathbb{R}}\rVert_{L^{2}(\mathbb{R})}=0,$$ by Plancherel's theorem, we get $$\lim_{\epsilon\to0}\lVert\mathscr{F}_{I}\left(f_{\epsilon}\right)-\mathscr{F}_{I}\left(f|_{\mathbb{R}}\right)\rVert_{L^{2}(\mathbb{R})}=0.$$ It follows that if 
		\begin{equation}\label{PW-LILI-lim-FT-fe}
			\lim_{\epsilon\to0}\frac{1}{\sqrt{2\pi}}\int_{\mathbb{R}}e^{-Iuw}f_{\epsilon}\left(u\right)\mathrm{d}u=0
		\end{equation}
		holds for $w\in\mathbb{R}$ and $|w|>B$, then $\mathscr{F}_{I}\left(f|_{\mathbb{R}}\right)$ vanishes outside $\left[-B,B\right]$. It is easy to see that
		\begin{equation*}
			f\left(u\right)=\frac{1}{\sqrt{2\pi}}\int_{-B}^{B}e^{Iuw}\mathscr{F}_{I}\left(f|_{\mathbb{R}}\right)\left(w\right)\mathrm{d}w.
		\end{equation*}
		Since each side of \eqref{PW-LILI-HFT} is a holomorphic function, we can conclude that \eqref{PW-LILI-HFT} holds for every $z=u+Iv\in L_{I}$. 
		
		Now, one can show that \eqref{PW-LILI-lim-FT-fe} holds for $w\in\mathbb{R}$ and $|w|>B$. For each $\delta\in\mathbb{R}$, let 
		$$\mathscr{L}_{\delta}\lr{s}=se^{I\delta},0\leq s<\infty,$$ and put 
		\begin{equation}\label{PW-LILI-norm-xi}
			\mathscr{O}_{\delta}:=\LR{\xi:\Re\LR{\xi e^{I\delta}}>B}.
		\end{equation}
		If $\xi\in\mathscr{O}_{\delta}$, we define 
		\[\Phi_{\delta}\lr{\xi}=\int_{\mathscr{L}_{\delta}}e^{-\xi z}f\lr{z}\mathrm{d}z=e^{I\delta}\int_{0}^{\infty}e^{-\xi se^{I\delta}}f\lr{se^{I\delta}}\mathrm{d}s.\] 
		By \eqref{PW-LILI-norm-f} and \eqref{PW-LILI-norm-xi}, for $z\in\mathscr{L}_{\delta}$, we get
		\[\abs{e^{-\xi z}f\lr{z}}\leq Ce^{\Lr{B-\Re\LR{\xi e^{I\delta}}}s}.\]
		Using Fubini's theorem and Morera' theorem, we obtain that $\Phi_{\delta}$ is holomorphic on $\mathscr{O}_{\delta}$.
		
		Next, we will show that if $0<\delta_{2}-\delta_{1}<\pi$, then $\Phi_{\delta_{1}}$ and $\Phi_{\delta_{2}}$ coincide in $\mathscr{O}_{\delta_{1}}\cap\mathscr{O}_{\delta_{2}}$. In fact, we can have this conclusion by the same argument used in \cite{rudin1987real}. It follows that $\Phi_{\delta_{1}}$ and $\Phi_{\delta_{2}}$ are analytic continuations of each other. Hence we have 
		\begin{align*}
			\int_{\mathbb{R}}e^{-Iuw}f_{\epsilon}\lr{u}\mathrm{d}u=&\Phi_{0}\lr{\epsilon+Iw}-\Phi_{\pi}\lr{-\epsilon+Iw}\\
			=&\Phi_{-\frac{\pi}{2}}\lr{\epsilon+Iw}-\Phi_{-\frac{\pi}{2}}\lr{-\epsilon+Iw}
		\end{align*}
		when $w>B$ and
		\begin{align*}
			\int_{\mathbb{R}}e^{-Iuw}f_{\epsilon}\lr{u}\mathrm{d}u=&\Phi_{0}\lr{\epsilon+Iw}-\Phi_{\pi}\lr{-\epsilon+Iw}\\
			=&\Phi_{\frac{\pi}{2}}\lr{\epsilon+Iw}-\Phi_{\frac{\pi}{2}}\lr{-\epsilon+Iw}
		\end{align*}
		when $w<-B$. It is obvious that \eqref{PW-LILI-lim-FT-fe} holds for $w\in\mathbb{R}$ and $|w|>B$.
		
		Assume that (ii) holds. For any fixed $I\in\mathbb{S}$, let
		\begin{equation*}
			g\left(u+Iv\right)=\frac{1}{\sqrt{2\pi}}\int_{-B}^{B}e^{I\left(u+Iv\right)t}\mathscr{F}_{I}\left(f|_{\mathbb{R}}\right)\left(t\right)\mathrm{d}t.
		\end{equation*}
		It is obvious that $g\left(u+Iv\right)$ is holomorphic on $L_{I}$. Using the Cauchy–Schwarz inequality and Plancherel's theorem, we have
		\begin{align*}
			\lvert g\left(u+Iv\right)\rvert&=\Big\lvert\frac{1}{\sqrt{2\pi}}\int_{-B}^{B}e^{I\left(u+Iv\right)t}\mathscr{F}_{I}\left(f|_{\mathbb{R}}\right)\left(t\right)\mathrm{d}t\Big\rvert\\
			&\leq\frac{1}{\sqrt{2\pi}}\int_{-B}^{B}\lvert e^{I\left(u+Iv\right)t}\mathscr{F}_{I}\left(f|_{\mathbb{R}}\right)\left(t\right)\rvert\mathrm{d}t\\
			&\leq\frac{1}{\sqrt{2\pi}}\left(\int_{-B}^{B}\Big\lvert e^{I\left(u+Iv\right)t}\Big\rvert^{2}\mathrm{d}t\right)^{\frac{1}{2}}\left(\int_{-B}^{B}\lvert \mathscr{F}_{I}\left(f|_{\mathbb{R}}\right)\left(t\right)\rvert^{2}\mathrm{d}t\right)^{\frac{1}{2}}\\
			&\leq\frac{\sqrt{2B}}{\sqrt{2\pi}}e^{B|v|}\lVert f|_{\mathbb{R}}\rVert_{L^{2}(\mathbb{R})}\leq Ce^{B|u+Iv|}.
		\end{align*}
		Since $f(u)=g(u)$ holds for $u\in\mathbb{R}$ and both functions are holomorphic on $L_{I}$, we conclude that $f\left(u+Iv\right)=g\left(u+Iv\right)$. It follows that $|f\left(z\right)|\leq Ce^{B|z|}$ for all $z=u+Iv\in L_{I}$.
		
		The proof is completed. 
	\end{proof}
	\begin{lemma}\label{PW-NN}
		Suppose $f\in\mathcal{N}_{l}\left(\mathbb{R}^{n+1}\right)$ with $f|_{\mathbb{R}}\in L^{2}(\mathbb{R})$, and $B,C>0$ are two positive numbers. There exists a unit purely imaginary paravector $I\in\mathbb{S}$, such that $|f\left(u+Iv\right)|\leq Ce^{B|u+Iv|}$ for all $u,v\in\mathbb{R}$. Then for every $J\in\mathbb{S}$, $\mathbf{supp}\mathscr{F}_{J}(f|_{\mathbb{R}})\subset\Lr{-B,B}$ and
		\begin{equation*}
			f\left(u+Jv\right)=\frac{1}{\sqrt{2\pi}}\int_{-B}^{B}e^{J\left(u+Jv\right)t}\mathscr{F}_{J}\left(f|_{\mathbb{R}}\right)\left(t\right)\mathrm{d}t
		\end{equation*}
		for all $u,v\in\mathbb{R}$.
	\end{lemma}
	\begin{proof}
		For given $I\in\mathbb{S}$, by $f\in\mathcal{N}_{l}\left(\mathbb{R}^{n+1}\right)$, we get $f\left(L_{I}\right)\subseteq L_{I}$ and $f\left(u-Iv\right)=\overline{f\left(u+Iv\right)}$. Since $|f\left(u+Iv\right)|\leq Ce^{B|u+Iv|}$ for all $u,v\in\mathbb{R}$ and $f|_{\mathbb{R}}\in L^{2}(\mathbb{R})$, by Lemma \ref{PW-LILI}, we have $\mathbf{supp}\mathscr{F}_{I}(f|_{\mathbb{R}})\subset\left[-B,B\right]$ and
		\begin{equation}\label{PW-NN-f}
			f\left(u+Iv\right)=\frac{1}{\sqrt{2\pi}}\int_{-B}^{B}e^{I\left(u+Iv\right)t}\mathscr{F}_{I}\left(f|_{\mathbb{R}}\right)\left(t\right)\mathrm{d}t.
		\end{equation}
		Taking the conjugation to both sides of \eqref{PW-NN-f}, we get
		\begin{align}\label{PW-NN-cf0}
			\overline{f\left(u+Iv\right)}=&\overline{\frac{1}{\sqrt{2\pi}}\int_{-B}^{B}e^{I\left(u+Iv\right)t}\mathscr{F}_{I}\left(f|_{\mathbb{R}}\right)\left(t\right)\mathrm{d}t}\notag\\
			=&\frac{1}{\sqrt{2\pi}}\int_{-B}^{B}e^{-I\left(u-Iv\right)t}\overline{\mathscr{F}_{I}\left(f|_{\mathbb{R}}\right)\left(t\right)}\mathrm{d}t\notag\\
			=&\frac{1}{\sqrt{2\pi}}\int_{-B}^{B}e^{I\left(u-Iv\right)s}\overline{\mathscr{F}_{I}\left(f|_{\mathbb{R}}\right)\left(-s\right)}\mathrm{d}s.
		\end{align}
		Using the fact that $f\left(u-Iv\right)=\overline{f\left(u+Iv\right)}$, we have  $f\left(u\right)=\overline{f\left(u\right)}$ for $v=0$. It means that $f|_{\mathbb{R}}$ is real-valued. It follows that
		\begin{align}\label{PW-NN-XTFT}
			\overline{\frac{1}{\sqrt{2\pi}}\int_{-T}^{T}e^{-Iu\left(-w\right)}f\left(u\right)\mathrm{d}u}&=\frac{1}{\sqrt{2\pi}}\int_{-T}^{T}e^{-Iuw}\overline{f\left(u\right)}\mathrm{d}u\notag\\
			&=\frac{1}{\sqrt{2\pi}}\int_{-T}^{T}e^{-Iuw}f\left(u\right)\mathrm{d}u.
		\end{align}
		Since $$\lim_{T\to\infty}\Big\lVert \mathscr{F}_{I}\left(f|_{\mathbb{R}}\right)\left(t\right)-\frac{1}{\sqrt{2\pi}}\int_{-T}^{T}e^{-Iuw}f\left(u\right)\mathrm{d}u\Big\rVert_{L^{2}(\mathbb{R})}=0$$ and \eqref{PW-NN-XTFT}, we have 
		\begin{equation}\label{PW-NN-CFT}
			\overline{\mathscr{F}_{I}\left(f|_{\mathbb{R}}\right)\left(-s\right)}=\mathscr{F}_{I}\left(f|_{\mathbb{R}}\right)\left(s\right).
		\end{equation}
		Substituting \eqref{PW-NN-CFT} into \eqref{PW-NN-cf0}, we obtain
		\begin{equation*}\label{PW-NN-cf}
			\overline{f\left(u+Iv\right)}=\frac{1}{\sqrt{2\pi}}\int_{-B}^{B}e^{I\left(u-Iv\right)s}\mathscr{F}_{I}\left(f|_{\mathbb{R}}\right)\left(s\right)\mathrm{d}s=f\left(u-Iv\right).
		\end{equation*}
		Since $f\left(L_{I}\right)\subseteq L_{I}$, $f$ can be written as $f=\Re\left\{f\right\}+I\Im\left\{f\right\}$ where $\Re\left\{f\right\}$ and $\Im\left\{f\right\}$ are real-valued. It follows that 
		\begin{align}\label{PW-NN-2Ref}
			\Re\left\{f\left(u+Iv\right)\right\}=&\frac{1}{2}\lr{\overline{f\left(u+Iv\right)}+f\left(u+Iv\right)}\notag\\
			=&\frac{1}{2\sqrt{2\pi}}\int_{-B}^{B}e^{I\left(u-Iv\right)t}\mathscr{F}_{I}\left(f|_{\mathbb{R}}\right)\left(t\right)\mathrm{d}t\notag\\
			&+\frac{1}{2\sqrt{2\pi}}\int_{-B}^{B}e^{I\left(u+Iv\right)t}\mathscr{F}_{I}\left(f|_{\mathbb{R}}\right)\left(t\right)\mathrm{d}t\notag\\
			=&\frac{1}{2\sqrt{2\pi}}\int_{-B}^{B}\left[\cos\left(ut\right)+I\sin\left(ut\right)\right]\left(e^{vt}+e^{-vt}\right)\mathscr{F}_{I}\left(f|_{\mathbb{R}}\right)\left(t\right)\mathrm{d}t
		\end{align}
		and
		\begin{align}\label{PW-NN--2IImf}
			\Im\left\{f\left(u+Iv\right)\right\}=&\frac{I}{2}\lr{\overline{f\left(u+Iv\right)}-f\left(u+Iv\right)}\notag\\
			=&\frac{I}{2\sqrt{2\pi}}\int_{-B}^{B}\left[\cos\left(ut\right)+I\sin\left(ut\right)\right]\left(e^{vt}-e^{-vt}\right)\mathscr{F}_{I}\left(f|_{\mathbb{R}}\right)\left(t\right)\mathrm{d}t.
		\end{align}
		Notice that for given $T>0$,
		\begin{align*}
			&\frac{1}{\sqrt{2\pi}}\int_{-T}^{T}e^{-Iuw}f\left(u\right)\mathrm{d}u\\
			=&\frac{1}{\sqrt{2\pi}}\int_{-T}^{T}\cos\left(uw\right)f\left(u\right)\mathrm{d}u-I\frac{1}{\sqrt{2\pi}}\int_{-T}^{T}\sin\left(uw\right)f\left(u\right)\mathrm{d}u
		\end{align*}
		where $$\frac{1}{\sqrt{2\pi}}\int_{-T}^{T}\cos\left(uw\right)f\left(u\right)\mathrm{d}u$$ and $$\frac{1}{\sqrt{2\pi}}\int_{-T}^{T}\sin\left(uw\right)f\left(u\right)\mathrm{d}u$$ are real-valued functions. Since $$\lim_{T\to\infty}\Big\lVert \mathscr{F}_{I}\left(f|_{\mathbb{R}}\right)\left(w\right)-\frac{1}{\sqrt{2\pi}}\int_{-T}^{T}e^{-Iuw}f\left(u\right)\mathrm{d}u\Big\rVert_{L^{2}(\mathbb{R})}=0,$$ we have that $\mathscr{F}_{I}\left(f|_{\mathbb{R}}\right)$ is $L_{I}$-valued. So, $\mathscr{F}_{I}\left(f|_{\mathbb{R}}\right)$ can be written as
		\begin{equation}\label{PW-NN-FTLI}
			\mathscr{F}_{I}\left(f|_{\mathbb{R}}\right)=\Re\left\{\mathscr{F}_{I}\left(f|_{\mathbb{R}}\right)\right\}+I\Im\left\{\mathscr{F}_{I}\left(f|_{\mathbb{R}}\right)\right\}.
		\end{equation}
		It is easy to see that
		\begin{equation*}
			\lim_{T\to\infty}\Big\lVert \Re\left\{\mathscr{F}_{I}\left(f|_{\mathbb{R}}\right)\lr{w}\right\}-\frac{1}{\sqrt{2\pi}}\int_{-T}^{T}\cos\lr{uw}f\left(u\right)\mathrm{d}u\Big\rVert_{L^{2}(\mathbb{R})}=0
		\end{equation*}
		and
		\begin{equation*}
			\lim_{T\to\infty}\Big\lVert \Im\left\{\mathscr{F}_{I}\left(f|_{\mathbb{R}}\right)\lr{w}\right\}+\frac{1}{\sqrt{2\pi}}\int_{-T}^{T}\sin\lr{uw}f\left(u\right)\mathrm{d}u\Big\rVert_{L^{2}(\mathbb{R})}=0.
		\end{equation*}
		It means that $\Re\left\{\mathscr{F}_{I}\left(f|_{\mathbb{R}}\right)\right\}$ and $\Im\left\{\mathscr{F}_{I}\left(f|_{\mathbb{R}}\right)\right\}$ are independent of $I$. In other words, for every $I,J\in\mathbb{S}$, we have $\Re\left\{\mathscr{F}_{I}\left(f|_{\mathbb{R}}\right)\right\}=\Re\left\{\mathscr{F}_{J}\left(f|_{\mathbb{R}}\right)\right\}$ and $\Im\left\{\mathscr{F}_{I}\left(f|_{\mathbb{R}}\right)\right\}=\Im\left\{\mathscr{F}_{J}\left(f|_{\mathbb{R}}\right)\right\}$.
		Substituting \eqref{PW-NN-FTLI} into \eqref{PW-NN-2Ref} and \eqref{PW-NN--2IImf}, we obtain
		\begin{align*}
			&\Re\left\{f\left(u+Iv\right)\right\}\\
			=&\frac{1}{2\sqrt{2\pi}}\int_{-B}^{B}\left[\cos\left(ut\right)+I\sin\left(ut\right)\right]\left(e^{vt}+e^{-vt}\right)\left[\Re\left\{\mathscr{F}_{I}\left(f|_{\mathbb{R}}\right)\left(t\right)\right\}+I\Im\left\{\mathscr{F}_{I}\left(f|_{\mathbb{R}}\right)\left(t\right)\right\}\right]\mathrm{d}t\\
			=&\frac{1}{2\sqrt{2\pi}}\int_{-B}^{B}\left(e^{vt}+e^{-vt}\right)\left[\cos\left(ut\right)\Re\left\{\mathscr{F}_{I}\left(f|_{\mathbb{R}}\right)\left(t\right)\right\}-\sin\left(ut\right)\Im\left\{\mathscr{F}_{I}\left(f|_{\mathbb{R}}\right)\left(t\right)\right\}\right]\mathrm{d}t\\
			&+I\frac{1}{2\sqrt{2\pi}}\int_{-B}^{B}\left(e^{vt}+e^{-vt}\right)\left[\cos\left(ut\right)\Im\left\{\mathscr{F}_{I}\left(f|_{\mathbb{R}}\right)\left(t\right)\right\}+\sin\left(ut\right)\Re\left\{\mathscr{F}_{I}\left(f|_{\mathbb{R}}\right)\left(t\right)\right\}\right]\mathrm{d}t
		\end{align*}
		and
		\begin{align*}
			&\Im\left\{f\left(u+Iv\right)\right\}\\
			=&\frac{I}{2\sqrt{2\pi}}\int_{-B}^{B}\left[\cos\left(ut\right)+I\sin\left(ut\right)\right]\left(e^{vt}-e^{-vt}\right)\left[\Re\left\{\mathscr{F}_{I}\left(f|_{\mathbb{R}}\right)\left(t\right)\right\}+I\Im\left\{\mathscr{F}_{I}\left(f|_{\mathbb{R}}\right)\left(t\right)\right\}\right]\mathrm{d}t\\
			=&\frac{I}{2\sqrt{2\pi}}\int_{-B}^{B}\left(e^{vt}-e^{-vt}\right)\left[\cos\left(ut\right)\Re\left\{\mathscr{F}_{I}\left(f|_{\mathbb{R}}\right)\left(t\right)\right\}-\sin\left(ut\right)\Im\left\{\mathscr{F}_{I}\left(f|_{\mathbb{R}}\right)\left(t\right)\right\}\right]\mathrm{d}t\\
			&-\frac{1}{2\sqrt{2\pi}}\int_{-B}^{B}\left(e^{vt}-e^{-vt}\right)\left[\cos\left(ut\right)\Im\left\{\mathscr{F}_{I}\left(f|_{\mathbb{R}}\right)\left(t\right)\right\}+\sin\left(ut\right)\Re\left\{\mathscr{F}_{I}\left(f|_{\mathbb{R}}\right)\left(t\right)\right\}\right]\mathrm{d}t
		\end{align*}
		respectively. Using the fact that $\Re\left\{f\right\}$ and $\Im\left\{f\right\}$ are real-valued function, we get 
		\begin{equation*}
			\int_{-B}^{B}\left(e^{vt}+e^{-vt}\right)\left[\cos\left(ut\right)\Im\left\{\mathscr{F}_{I}\left(f|_{\mathbb{R}}\right)\left(t\right)\right\}+\sin\left(ut\right)\Re\left\{\mathscr{F}_{I}\left(f|_{\mathbb{R}}\right)\left(t\right)\right\}\right]\mathrm{d}t=0,
		\end{equation*}
		and
		\begin{equation*}
			\int_{-B}^{B}\left(e^{vt}-e^{-vt}\right)\left[\cos\left(ut\right)\Re\left\{\mathscr{F}_{I}\left(f|_{\mathbb{R}}\right)\left(t\right)\right\}-\sin\left(ut\right)\Im\left\{\mathscr{F}_{I}\left(f|_{\mathbb{R}}\right)\left(t\right)\right\}\right]\mathrm{d}t=0.
		\end{equation*}
		It follows that
		\begin{equation*}
			\Re\left\{f\left(u+Iv\right)\right\}=\frac{1}{\sqrt{2\pi}}\int_{-B}^{B}e^{-vt}\left[\cos\left(ut\right)\Re\left\{\mathscr{F}_{I}\left(f|_{\mathbb{R}}\right)\left(t\right)\right\}-\sin\left(ut\right)\Im\left\{\mathscr{F}_{I}\left(f|_{\mathbb{R}}\right)\left(t\right)\right\}\right]\mathrm{d}t
		\end{equation*}
		and
		\begin{equation*}
			\Im\left\{f\left(u+Iv\right)\right\}=\frac{1}{\sqrt{2\pi}}\int_{-B}^{B}e^{-vt}\left[\cos\left(ut\right)\Im\left\{\mathscr{F}_{I}\left(f|_{\mathbb{R}}\right)\left(t\right)\right\}+\sin\left(ut\right)\Re\left\{\mathscr{F}_{I}\left(f|_{\mathbb{R}}\right)\left(t\right)\right\}\right]\mathrm{d}t.
		\end{equation*}
		By Proposition \ref{prop-RF}, for every $J\in\mathbb{S}$, we have
		\begin{align*}
			&f\left(u+Jv\right)\\
			=&\frac{1}{2}\left[f\left(u-Iv\right)+f\left(u+Iv\right)\right]+\frac{JI}{2}\left[f\left(u-Iv\right)-f\left(u+Iv\right)\right]\\
			=&\frac{1}{2}\left[\overline{f\left(u+Iv\right)}+f\left(u+Iv\right)\right]+\frac{JI}{2}\left[\overline{f\left(u+Iv\right)}-f\left(u+Iv\right)\right]\\
			=&\Re\left\{f\left(u+Iv\right)\right\}+J\Im\left\{f\left(u+Iv\right)\right\}\\
			=&\frac{1}{\sqrt{2\pi}}\int_{-B}^{B}e^{-vt}\left[\cos\left(ut\right)\Re\left\{\mathscr{F}_{I}\left(f|_{\mathbb{R}}\right)\left(t\right)\right\}-\sin\left(ut\right)\Im\left\{\mathscr{F}_{I}\left(f|_{\mathbb{R}}\right)\left(t\right)\right\}\right]\mathrm{d}t\\
			&+J\frac{1}{\sqrt{2\pi}}\int_{-B}^{B}e^{-vt}\left[\cos\left(ut\right)\Im\left\{\mathscr{F}_{I}\left(f|_{\mathbb{R}}\right)\left(t\right)\right\}+\sin\left(ut\right)\Re\left\{\mathscr{F}_{I}\left(f|_{\mathbb{R}}\right)\left(t\right)\right\}\right]\mathrm{d}t\\
			=&\frac{1}{\sqrt{2\pi}}\int_{-B}^{B}e^{-vt}\Big\{\left[\cos\left(ut\right)+J\sin\left(ut\right)\right]\Re\left\{\mathscr{F}_{I}\left(f|_{\mathbb{R}}\right)\left(t\right)\right\}\\
			&+\left[\cos\left(ut\right)+J\sin\left(ut\right)\right]J\Im\left\{\mathscr{F}_{I}\left(f|_{\mathbb{R}}\right)\left(t\right)\right\}\Big\}\mathrm{d}t\\
			=&\frac{1}{\sqrt{2\pi}}\int_{-B}^{B}e^{-vt}e^{Jut}\left[\Re\left\{\mathscr{F}_{I}\left(f|_{\mathbb{R}}\right)\left(t\right)\right\}+J\Im\left\{\mathscr{F}_{I}\left(f|_{\mathbb{R}}\right)\left(t\right)\right\}\right]\mathrm{d}t\\
			=&\frac{1}{\sqrt{2\pi}}\int_{-B}^{B}e^{J\left(u+Jv\right)t}\mathscr{F}_{J}\left(f|_{\mathbb{R}}\right)\left(t\right)\mathrm{d}t
		\end{align*}
		and $\lvert\mathscr{F}_{J}\left(f|_{\mathbb{R}}\right)(t)\rvert^{2}=\lvert\mathscr{F}_{I}\left(f|_{\mathbb{R}}\right)(t)\rvert^{2}$. Thus, we have $\mathbf{supp}\mathscr{F}_{J}\left(f|_{\mathbb{R}}\right)=\mathbf{supp}\mathscr{F}_{I}\left(f|_{\mathbb{R}}\right)\subset\Lr{-B,B}$.
	\end{proof}
	
	Now, we can prove Theorem \ref{PW-Rn+1Rn}.
	\begin{proof}[Proof of Theorem \ref{PW-Rn+1Rn}]
		Assume that (i) holds. Let $I_{1}, \dots, I_{n}$ be an orthonormal basis of $\mathbb{R}_{n}$. Since $f$ is a left slice monogenic function, using Proposition \ref{prop-NE}, for every $J\in\mathbb{S}$, we have
		\begin{equation}\label{PW-Rn+1Rn-f}
			f\left(u+Jv\right)=\sum_{A}h_{A}\left(u+Jv\right)I_{A}
		\end{equation}
		where $h_{A}\in\mathcal{N}_{l}$ and $A\in\mathcal{P}_{<}\lr{\left\{1,\dots,n\right\}}$. Since $\lvert f\left(u+Jv\right)\rvert\leq Ce^{A|u+Jv|}$, by \eqref{norm-hA}, we obtain $\lvert h_{A}\left(u+Jv\right)\rvert\leq Ce^{B|u+Jv|}$. Since $\abs{f|_{\mathbb{R}}}^2=\sum_{A}\abs{h_{A}|_{\mathbb{R}}}^2$ and $f|_{\mathbb{R}}\in L^{2}(\mathbb{R})$, we get $h_{A}|_{\mathbb{R}}\in L^{2}(\mathbb{R})$. Using Lemma \ref{PW-NN}, we have $\mathbf{supp}\mathscr{F}_{J}\left(h_{A}|_{\mathbb{R}}\right)\subset\Lr{-B,B}$ and 
		\begin{equation}\label{PW-Rn+1Rn-hA}
			h_{A}\left(u+Jv\right)=\frac{1}{\sqrt{2\pi}}\int_{-B}^{B}e^{J\left(u+Jv\right)t}\mathscr{F}_{J}\left(h_{A}|_{\mathbb{R}}\right)\left(t\right)\mathrm{d}t
		\end{equation}
		Combining \eqref{PW-Rn+1Rn-f} and \eqref{PW-Rn+1Rn-hA}, we have
		\begin{align*}
			f\left(u+Jv\right)=&\sum_{A}\Big[\frac{1}{\sqrt{2\pi}}\int_{-B}^{B}e^{J\left(u+Jv\right)t}\mathscr{F}_{J}\left(h_{A}|_{\mathbb{R}}\right)\left(t\right)\mathrm{d}t\Big]I_{A}\\
			=&\frac{1}{\sqrt{2\pi}}\int_{-B}^{B}e^{J\left(u+Jv\right)t}\sum_{A}\mathscr{F}_{J}\left(h_{A}|_{\mathbb{R}}\right)\left(t\right)I_{A}\mathrm{d}t\\
			=&\frac{1}{\sqrt{2\pi}}\int_{-B}^{B}e^{J\left(u+Jv\right)t}\mathscr{F}_{J}\left(\sum_{A}h_{A}|_{\mathbb{R}}I_{A}\right)\left(t\right)\mathrm{d}t\\
			=&\frac{1}{\sqrt{2\pi}}\int_{-B}^{B}e^{J\left(u+Jv\right)t}\mathscr{F}_{J}\left(f|_{\mathbb{R}}\right)\left(t\right)\mathrm{d}t
		\end{align*}
		and $\mathscr{F}_{J}\left(f|_{\mathbb{R}}\right)\left(t\right)=\sum_{A}\mathscr{F}_{J}\left(h_{A}|_{\mathbb{R}}\right)\left(t\right)I_{A}$. Thus, we have $\mathbf{supp}\mathscr{F}_{J}\left(f|_{\mathbb{R}}\right)=\bigcup_{A}\mathbf{supp}\mathscr{F}_{I}(h_{A})\subset\Lr{-B,B}$.
		
		Assume that (ii) holds. For given $I\in\mathbb{S}$, let
		\begin{equation*}
			g\left(u+Iv\right)=\frac{1}{\sqrt{2\pi}}\int_{-B}^{B}e^{I\left(u+Iv\right)t}\mathscr{F}_{I}\left(f|_{\mathbb{R}}\right)\left(t\right)\mathrm{d}t.
		\end{equation*}
		It is easy to see that $g\left(u+Iv\right)$ is holomorphic on $L_{I}$. Using the Cauchy–Schwarz inequality and Plancherel's theorem, we get
		\begin{align*}
			\lvert g\left(u+Iv\right)\rvert&=\Big\lvert\frac{1}{\sqrt{2\pi}}\int_{-B}^{B}e^{I\left(u+Iv\right)t}\mathscr{F}_{I}\left(f|_{\mathbb{R}}\right)\left(t\right)\mathrm{d}t\Big\rvert\\
			&\leq\frac{1}{\sqrt{2\pi}}\int_{-B}^{B}\lvert e^{I\left(u+Iv\right)t}\mathscr{F}_{I}\left(f|_{\mathbb{R}}\right)\left(t\right)\rvert\mathrm{d}t\\
			&\leq\frac{1}{\sqrt{2\pi}}\left(\int_{-B}^{B}\Big\lvert e^{I\left(u+Iv\right)t}\Big\rvert^{2}\mathrm{d}t\right)^{\frac{1}{2}}\left(\int_{-B}^{B}\lvert \mathscr{F}_{I}\left(f|_{\mathbb{R}}\right)\left(t\right)\rvert^{2}\mathrm{d}t\right)^{\frac{1}{2}}\\
			&\leq\frac{\sqrt{2B}}{\sqrt{2\pi}}e^{B|v|}\lVert f\rVert_{L^{2}(\mathbb{R})}\leq Ce^{B|u+Iv|}.
		\end{align*}
		Since $f\left(u\right)=g\left(u\right)$ and both functions are holomorphic on $L_{I}$, we infer that $f\left(u+Iv\right)=g\left(u+Iv\right)$. By Proposition \ref{prop-RF}, for every $J\in\mathbb{S}$, we have
		\begin{equation*}
			f\left(u+Jv\right)=\frac{1+JI}{2}f\left(u-Iv\right)+\frac{1-JI}{2}f\left(u+Iv\right).
		\end{equation*}
		Hence, we get
		\begin{align*}
			\lvert f\left(u+Jv\right)\rvert&\leq\Big\lvert f\left(u-Iv\right)\Big\rvert+\Big\lvert f\left(u+Iv\right)\Big\rvert\\
			&\leq Ce^{A|u-Iv|}+Ce^{A|u+Iv|}\\
			&=2Ce^{A|u+Iv|}\\
			&=C'e^{A|u+Iv|}.
		\end{align*}
		
		The proof is completed.
	\end{proof}
	Moreover, as an application, we can give the reproducing kernel for the left slice monogenic Paley-Wiener space. In the following, for $B>0$, we define the left slice monogenic Paley-Wiener space as 
	\begin{align*}
		&PW_{B}^{2}\lr{\mathbb{R}^{n+1}}\\
		=&\LR{f\in\mathcal{SM}_{l}\lr{\mathbb{R}^{n+1}},f|_{\mathbb{R}}\in L^{2}(\mathbb{R}),\supp\FT{I}{f|_{\mathbb{R}}}\subset\Lr{-B,B}\ \text{for}\ \text{every}\ I\in\mathbb{S}}.
	\end{align*}
	
	For any fixed $I\in\mathbb{S}$ and $\xi+I\eta\in L_{I}$,  we define 
	\begin{equation*}
		\mathscr{K}_{B,L_{I}}\lr{u+Iv,\xi+I\eta}=\frac{1}{2\pi}\int_{-B}^{B}e^{I\Lr{\lr{u+Iv}-\lr{\xi-I\eta}}w}\mathrm{d}w,
	\end{equation*}
	for all $u+Iv\in L_{I}.$ By Proposition \ref{prop-ext}, for every $J\in\mathbb{S}$, we define
	\begin{align}\label{RK-KB}
		\begin{split}
			&\mathscr{K}_{B,\mathbb{R}^{n+1}}\lr{u+Jv,\xi+I\eta}\\
			=&\ \ext\lr{\mathscr{K}_{B,L_{I}}\lr{u+Iv,\xi+I\eta}}\lr{u+Jv}\\
			=&\frac{1}{2}\Lr{\frac{1}{2\pi}\int_{-B}^{B}e^{I\Lr{\lr{u-Iv}-\lr{\xi-I\eta}}w}\mathrm{d}w+\frac{1}{2\pi}\int_{-B}^{B}e^{I\Lr{\lr{u+Iv}-\lr{\xi-I\eta}}w}\mathrm{d}w}\\
			&+\frac{JI}{2}\Lr{\frac{1}{2\pi}\int_{-B}^{B}e^{I\Lr{\lr{u-Iv}-\lr{\xi-I\eta}}w}\mathrm{d}w-\frac{1}{2\pi}\int_{-B}^{B}e^{I\Lr{\lr{u+Iv}-\lr{\xi-I\eta}}w}\mathrm{d}w}\\
			=&\frac{1}{2\pi}\int_{-B}^{B}\frac{1}{2}\Lr{e^{I\lr{u-Iv}w}+e^{I\lr{u+Iv}w}}e^{-I\lr{\xi-I\eta}w}\mathrm{d}w\\
			&+\frac{1}{2\pi}\int_{-B}^{B}\frac{JI}{2}\Lr{e^{I\lr{u-Iv}w}-e^{I\lr{u+Iv}w}}e^{-I\lr{\xi-I\eta}w}\mathrm{d}w\\
			=&\frac{1}{2\pi}\int_{-B}^{B}\ext\lr{e^{I\lr{u+Iv}w}}\lr{u+Jv}e^{-I\lr{\xi-I\eta}w}\mathrm{d}w.
		\end{split}
	\end{align}
	When $\eta=0$, it follows from \eqref{RK-KB} that
	\begin{align*}
		\mathscr{K}_{B,\mathbb{R}^{n+1}}\lr{u+Jv,\xi}=&\frac{1}{2}\Lr{\frac{1}{2\pi}\int_{-B}^{B}e^{I\lr{u-Iv-\xi}w}\mathrm{d}w+\frac{1}{2\pi}\int_{-B}^{B}e^{I\lr{u+Iv-\xi}w}\mathrm{d}w}\\
		&+\frac{JI}{2}\Lr{\frac{1}{2\pi}\int_{-B}^{B}e^{I\lr{u-Iv-\xi}w}\mathrm{d}w-\frac{1}{2\pi}\int_{-B}^{B}e^{I\lr{u+Iv-\xi}w}\mathrm{d}w}.
	\end{align*}
	Notice that
	\begin{equation*}
		\overline{\frac{1}{2\pi}\int_{-B}^{B}e^{I\lr{u-Iv-\xi}w}\mathrm{d}w}=\frac{1}{2\pi}\int_{-B}^{B}e^{I\lr{u+Iv-\xi}w}\mathrm{d}w.
	\end{equation*}
	Hence, we get
	\begin{align*}
		\mathscr{K}_{B,\mathbb{R}^{n+1}}\lr{u+Jv,\xi}=&\Re\LR{\frac{1}{2\pi}\int_{-B}^{B}e^{I\lr{u+Iv-\xi}w}\mathrm{d}w}\\
		&+J\Im\LR{\frac{1}{2\pi}\int_{-B}^{B}e^{I\lr{u+Iv-\xi}w}\mathrm{d}w}\\
		=&\frac{1}{2\pi}\int_{-B}^{B}e^{J\lr{u-\xi+Jv}w}\mathrm{d}w.
	\end{align*}
	
	Now, we give the reproducing kernel formula of $PW_{B}^{2}\lr{\mathbb{R}^{n+1}}$ as follows.
	\begin{corollary}
		
		Let $f\in PW_{B}^{2}\lr{\mathbb{R}^{n+1}}$. Then for every $\mathbf{x}\in\mathbb{R}^{n+1}$, we have
		\begin{equation*}
			f\lr{\mathbf{x}}=\int_{\mathbb{R}}\mathscr{K}_{B,\mathbb{R}^{n+1}}\lr{\mathbf{x},\xi}f\lr{\xi}\mathrm{d}\xi.
		\end{equation*}
		where
		\begin{equation*}
			\mathscr{K}_{B,\mathbb{R}^{n+1}}\lr{\mathbf{x},\xi}=\frac{1}{2\pi}\int_{-B}^{B}e^{I_{\mathbf{X}}\lr{\mathbf{x}-\xi}w}\mathrm{d}w.
		\end{equation*}    
	\end{corollary}
	\begin{proof}
		Using Theorem \ref{PW-Rn+1Rn} and Fubini's theorem, for every $J\in\mathbb{S}$, we have
		\begin{align*}
			f\lr{u+Jv}=&\frac{1}{\sqrt{2\pi}}\int_{-B}^{B}e^{J\lr{u+Jv}w}\FT{J}{f|_{\mathbb{R}}}\lr{w}\mathrm{d}w\\
			=&\frac{1}{2\pi}\int_{\mathbb{R}}\int_{-B}^{B}e^{J\lr{u+Jv}w}e^{-J\xi w}\mathrm{d}wf\lr{\xi}\mathrm{d}\xi\\
			=&\int_{\mathbb{R}}\mathscr{K}_{B,\mathbb{R}^{n+1}}\lr{u+Jv,\xi}f\lr{\xi}\mathrm{d}\xi.
		\end{align*}
	\end{proof}
	\subsection{Analogue in Bergman space}
	
	In this part we will prove the Paley-Wiener theorem for the functions in the left slice monogenic Bergman space in the right half-space. For the analogous result in the settings of complex analysis and Clifford algebra, we refer to \cite{bekolle2003lecture,dang2020fourier}.\\
	Define
	\begin{align*}
		&\mathcal{B}_{\text{slice}}^{p}\lr{\Pi_{+,I}}\\
		=&\LR{f\in\mathcal{SM}_{l}\lr{\mathbb{R}_{+}^{n+1}}:\norm{f}_{\mathcal{B}_{\text{slice}}^{p}\lr{\Pi_{+,I}}}^{p}:=\int_{0}^{\infty}\int_{-\infty}^{\infty}\Abs{f\lr{u+Iv}}^{p}\mathrm{d}v\mathrm{d}u<\infty},
	\end{align*}
	where $1\leq p<\infty$. For every $I,J\in\mathbb{S}$, by Proposition \ref{prop-RF}, we have
	\begin{equation}\label{abs-f}
		\Abs{f\lr{u+Iv}}\leq \frac{\sqrt{2}}{2}\lr{\Abs{f\lr{u+Jv}}+\Abs{f\lr{u-Jv}}}.
	\end{equation}
	So, we obtain
	\begin{align*}
		\int_{0}^{\infty}\int_{-\infty}^{\infty}\Abs{f\lr{u+Iv}}^{p}\mathrm{d}v \mathrm{d}u\leq& 2^{-\frac{p}{2}}\int_{0}^{\infty}\int_{-\infty}^{\infty}\Abs{f\lr{u+Jv}}^{p}\mathrm{d}v \mathrm{d}u\\
		&+2^{-\frac{p}{2}}\int_{0}^{\infty}\int_{-\infty}^{\infty}\Abs{f\lr{u-Jv}}^{p}\mathrm{d}v\mathrm{d}u\\
		\leq& 2^{-\frac{p}{2}+1}\int_{0}^{\infty}\int_{-\infty}^{\infty}\Abs{f\lr{u+Jv}}^{p}\mathrm{d}v\mathrm{d}u.
	\end{align*}
	
	we immediately get the following results.
	\begin{proposition}
		Let $f\in\mathcal{SM}_{l}\lr{\mathbb{R}_{+}^{n+1}}$. Then for every $I,J\in\mathbb{S}$, $f\in\mathcal{B}_{\text{slice}}^{p}\lr{\Pi_{+,I}}$ if and only if $f\in\mathcal{B}_{\text{slice}}^{p}\lr{\Pi_{+,J}}$.
	\end{proposition}
	Now, we introduce the left slice monogenic Bergman space of $\mathbb{R}_{+}^{n+1}$:
	\begin{definition}
		If $f\in\mathcal{SM}_{l}\lr{\mathbb{R}_{+}^{n+1}}$ and satisfies
		\begin{equation*}
			\norm{f}_{\mathcal{B}_{\text{slice}}^{p}\lr{\mathbb{R}_{+}^{n+1}}}^{p}=\sup_{I\in\mathbb{S}}\norm{f}_{\mathcal{B}_{\text{slice}}^{p}\lr{\Pi_{+,I}}}^{p}<\infty,
		\end{equation*}
		then we say that $f$ belongs to the left slice monogenic Bergman space $\mathcal{B}_{\text{slice}}^{p}\lr{\mathbb{R}_{+}^{n+1}}$.
	\end{definition}
	
	We give some basic inequalities for function in $\mathcal{B}_{\text{slice}}^{p}\lr{\mathbb{R}_{+}^{n+1}}$.
	\begin{proposition}\label{prop-inequality}
		Let $p\in\left[1,\infty\right)$.
		\begin{itemize}
			\item[(i)] There exists a constant $C=C(p)>0$ such that for all $I\in\mathbb{S}$ and $u,v\in\mathbb{R}$ and for all $f\in\mathcal{B}_{\text{slice}}^{p}\lr{\mathbb{R}_{+}^{n+1}}$, the following inequality holds:
			\begin{equation*}
				\abs{f\lr{u+Iv}}\leq Cu^{-4+\frac{2}{p}}\norm{f}_{\mathcal{B}_{\text{slice}}^{p}\lr{\mathbb{R}_{+}^{n+1}}}.
			\end{equation*} 
			\item[(ii)] There exists a constant $C=C(p)>0$ such that for all $I\in\mathbb{S}$ and $u\in\lr{0,\infty}$ and for all $f\in\mathcal{B}_{\text{slice}}^{p}\lr{\mathbb{R}_{+}^{n+1}}$, the following inequality holds:
			\begin{equation*}
				\norm{f_{u}|_{I\mathbb{R}}}_{L^{p}\lr{\mathbb{R}}}\leq Cu^{-4+\frac{3}{p}}\norm{f}_{\mathcal{B}_{\text{slice}}^{p}\lr{\mathbb{R}_{+}^{n+1}}},
			\end{equation*}
			where $f_{u}\lr{Iv}=f\lr{u+Iv}$.
		\end{itemize}
	\end{proposition}
	
	\begin{proof}
		The proof is similar to that given in \cite{bekolle2003lecture}. For any fixed $I\in\mathbb{S}$. Let $\mathbf{x}_{0}=u_{0}+Iv_{0}\in\Pi_{+,I}$ and let $D\lr{\mathbf{x}_{0},r}$ denote the disc of center $\mathbf{x}_{0}$ and radius $r$. Note that
		\begin{align*}
			\abs{f\lr{u_{0}+Iv_{0}}}=&\Abs{\frac{4}{\pi u_{0}^{2}}\int_{D\lr{u_{0}+Iv_{0},\frac{u_{0}}{2}}}f\lr{u+Iv}\mathrm{d}v\mathrm{d}u}\\
			\leq&\frac{4}{\pi u_{0}^{2}}\int_{D\lr{u_{0}+Iv_{0},\frac{u_{0}}{2}}}\abs{f\lr{u+Iv}}\mathrm{d}v\mathrm{d}u\\
			\leq&\lr{\frac{4}{\pi u_{0}^{2}}}^{1+\frac{1}{q}}\lr{\int_{D\lr{u_{0}+Iv_{0},\frac{u_{0}}{2}}}\abs{f\lr{u+Iv}}^{p}\mathrm{d}v\mathrm{d}u}^{\frac{1}{p}}\\
			\leq&\lr{\frac{4}{\pi u_{0}^{2}}}^{1+\frac{1}{q}}\lr{\int_{\frac{u_{0}}{2}}^{\frac{3u_{0}}{2}}\int_{v_{0}-\frac{u_{0}}{2}}^{v_{0}+\frac{u_{0}}{2}}\abs{f\lr{u+Iv}}^{p}\mathrm{d}v\mathrm{d}u}^{\frac{1}{p}}\\
			\leq&\lr{\frac{4}{\pi u_{0}^{2}}}^{1+\frac{1}{q}}\norm{f}_{\mathcal{B}_{\text{slice}}^{p}\lr{\Pi_{+,I}}}
		\end{align*}
		where $\frac{1}{p}+\frac{1}{q}=1$. Hence, we have 
		\[\abs{f\lr{u+Iv}}\leq C'u^{-4+\frac{2}{p}}\norm{f}_{\mathcal{B}_{\text{slice}}^{p}\lr{\Pi_{+,I}}}.\]
		By \eqref{abs-f}, for all $J\in\mathbb{S}$, we have
		\begin{align*}
			\abs{f\lr{u+Jv}}\leq&\frac{\sqrt{2}}{2}\lr{\abs{f\lr{u+Iv}}+\abs{f\lr{u-Iv}}}\\
			\leq& \frac{\sqrt{2}}{2}C'u^{-4+\frac{2}{p}}\norm{f}_{\mathcal{B}_{\text{slice}}^{p}\lr{\Pi_{+,I}}}\\
			\leq& C''u^{-4+\frac{2}{p}}\norm{f}_{\mathcal{B}_{\text{slice}}^{p}\lr{\mathbb{R}_{+}^{n+1}}}.
		\end{align*}
		Similarly, we get
		\begin{equation*}
			\abs{f\lr{u_{0}+Iw}}^{p}\leq \lr{\frac{4}{\pi u_{0}^{2}}}^{p+\frac{p}{q}}\int_{\frac{u_{0}}{2}}^{\frac{3u_{0}}{2}}\int_{w-\frac{u_{0}}{2}}^{w+\frac{u_{0}}{2}}\abs{f\lr{u+Iv}}^{p}\mathrm{d}v\mathrm{d}u.
		\end{equation*}
		Then, integration with respect to $w$ gives
		\begin{align*}
			\norm{f\lr{u_{0}+I\cdot}}_{L^{p}\lr{\mathbb{R}}}^{p}\leq& \lr{\frac{4}{\pi u_{0}^{2}}}^{p+\frac{p}{q}}\int_{\mathbb{R}}\int_{\frac{u_{0}}{2}}^{\frac{3u_{0}}{2}}\int_{w-\frac{u_{0}}{2}}^{w+\frac{u_{0}}{2}}\abs{f\lr{u+Iv}}^{p}\mathrm{d}v\mathrm{d}u\mathrm{d}w\\
			=&\lr{\frac{4}{\pi u_{0}^{2}}}^{p+\frac{p}{q}}\int_{\frac{u_{0}}{2}}^{\frac{3u_{0}}{2}}\Lr{\int_{\mathbb{R}}\int_{\mathbb{R}}\chi_{\Lr{v-\frac{u_{0}}{2},v+\frac{u_{0}}{2}}}\lr{w}\mathrm{d}w\abs{f\lr{u+Iv}}^{p}\mathrm{d}v}\mathrm{d}u\\
			\leq&\lr{\frac{4}{\pi u_{0}^{2}}}^{p+\frac{p}{q}}u_{0}\int_{\frac{u_{0}}{2}}^{\frac{3u_{0}}{2}}\int_{\mathbb{R}}\abs{f\lr{u+Iv}}^{p}\mathrm{d}v\mathrm{d}u\\
			\leq& \lr{\frac{4}{\pi u_{0}^{2}}}^{p+\frac{p}{q}}u_{0}\norm{f}_{\mathcal{B}_{\text{slice}}^{p}\lr{\Pi_{+,I}}}^{p}
		\end{align*}
		where $\frac{1}{p}+\frac{1}{q}=1$. Thus, we have
		\[\norm{f\lr{u_{0}+I\cdot}}_{L^{p}\lr{\mathbb{R}}}\leq Cu_{0}^{-4+\frac{3}{p}}\norm{f}_{\mathcal{B}_{\text{slice}}^{p}\lr{\Pi_{+,I}}}.\]
		It follows from \eqref{abs-f} that 
		\begin{align}\label{f-Bp-Hp}
			\norm{f\lr{u_{0}+J\cdot}}_{L^{p}\lr{\mathbb{R}}}^{p}\leq&2^{1-\frac{p}{2}}\norm{f\lr{u+I\cdot}}_{L^{p}\lr{\mathbb{R}}}^{p}\notag\\
			\leq& C''u_{0}^{-4p+3}\norm{f}_{\mathcal{B}_{\text{slice}}^{p}\lr{\Pi_{+,I}}}^{p}\notag\\
			\leq&C''u_{0}^{-4p+3}\norm{f}_{\mathcal{B}_{\text{slice}}^{p}\lr{\mathbb{R}_{+}^{n+1}}}^{p}
		\end{align}
		for all $J\in\mathbb{S}$.
	\end{proof}
	By \eqref{f-Bp-Hp}, we get 
	$$\sup_{J\in\mathbb{S}}\norm{f\lr{u+J\cdot}}_{L^{p}\lr{\mathbb{R}}}^{p}\leq C'u^{-4p+3}\norm{f}_{\mathcal{B}_{\text{slice}}^{p}\lr{\mathbb{R}_{+}^{n+1}}}^{p}.$$ 
	Hence, we have that if $f\in\mathcal{B}_{\text{slice}}^{p}\lr{\mathbb{R}_{+}^{n+1}}$, then for $\epsilon>0$, $f\lr{\epsilon+\cdot}\in H_{\text{slice}}^{p}\lr{\mathbb{R}_{+}^{n+1}}$.
	\begin{theorem}\label{PW-B}
		For $1\leq p\leq2$, $\frac{1}{p}+\frac{1}{q}=1$. Let $f\in\mathcal{B}_{\text{slice}}^{p}\lr{\mathbb{R}_{+}^{n+1}}$. Then for every $\mathbf{x}\in\mathbb{R}_{+}^{n+1}$, there exists a function $g_{I_{\mathbf{x}}}\in L^{q}\lr{(-\infty,0),\lr{-w}^{-\frac{q}{p}}\mathrm{d}w}$, such that 
		\begin{equation*}
			f\lr{\mathbf{x}}=\frac{1}{\sqrt{2\pi}}\int_{-\infty}^{0}e^{\mathbf{x}w}g_{I_{\mathbf{x}}}\lr{w}\mathrm{d}w.
		\end{equation*}
		Moreover, for $1<p\leq2$,
		\begin{equation*}
			\Lr{\int_{-\infty}^{0}\lr{\frac{1}{-pw}}^{\frac{q}{p}}\abs{g_{I_{\mathbf{x}}}\lr{w}}^{q}\mathrm{d}w}^{\frac{1}{q}}\leq\norm{f}_{\mathcal{B}_{\text{slice}}^{p}\lr{\mathbb{R}_{+}^{n+1}}}<\infty,
		\end{equation*}
		and for $p=1$,
		\begin{equation*}
			\sup_{w\in\left(-\infty,0\right)}\frac{\abs{g_{I_{\mathbf{x}}}\lr{w}}}{-w}\leq\norm{f}_{\mathcal{B}_{\text{slice}}^{1}\lr{\mathbb{R}_{+}^{n+1}}}<\infty.
		\end{equation*}
	\end{theorem}
	\begin{proof}
		The Proposition \ref{prop-inequality} implies that for $f\in\mathcal{B}_{\text{slice}}^{p}\lr{\mathbb{R}_{+}^{n+1}}$ and $\delta>0$, the function $f_{\delta}=f\lr{\cdot+\delta}$ belongs to $H_{\text{slice}}^{p}\lr{\mathbb{R}_{+}^{n+1}}$. By Theorem \ref{PW-Rn+1Rn-Hp-FT}, for all $I\in\mathbb{S}$ and $u,v\in\mathbb{R}$, we have
		\begin{equation*}
			f_{\delta}\lr{u+Iv}=\frac{1}{\sqrt{2\pi}}\int_{-\infty}^{0}e^{\lr{u+Iv}w}\FT{I}{f_{\delta}|_{I\mathbb{R}}}\lr{w}\mathrm{d}w.
		\end{equation*}
		Using the fact that for all $\delta'>0$,
		\begin{equation*}
			\FT{I}{f_{\delta}|_{I\mathbb{R}}}\lr{w}=e^{-\delta' w}\FT{I}{f_{\delta+\delta'}|_{I\mathbb{R}}}\lr{w}
		\end{equation*}
		holds almost everywhere $w$. Note that
		\begin{equation*}
			e^{-\delta w}\FT{I}{f_{\delta}|_{I\mathbb{R}}}\lr{w}=e^{-\lr{\delta+\delta'}w}\FT{I}{f_{\delta+\delta'}|_{I\mathbb{R}}}\lr{w}
		\end{equation*}
		holds almost everywhere $w$. Thus, if we let $g_{I}\lr{w}=e^{-\delta w}\FT{I}{f_{\delta}|_{I\mathbb{R}}}\lr{w}$, which is independent of $\delta$, then we have   
		\begin{equation*}
			f_{\delta}\lr{u+Iv}=\frac{1}{\sqrt{2\pi}}\int_{-\infty}^{0}e^{\lr{u+Iv+\delta}w}g_{I}\lr{w}\mathrm{d}w.
		\end{equation*}
		Consequently,
		\begin{equation*}
			f\lr{u+Iv}=\frac{1}{\sqrt{2\pi}}\int_{-\infty}^{0}e^{\lr{u+Iv}w}g_{I}\lr{w}\mathrm{d}w.
		\end{equation*}
		Moreover, by Hausdorff-Young's inequality, we have
		\begin{equation*}
			\lr{\int_{-\infty}^{0}\abs{e^{uw}g_{I}\lr{w}}^{q}\mathrm{d}w}^{\frac{p}{q}}\leq\int_{\mathbb{R}}\abs{f\lr{u+Iv}}^{p}\mathrm{d}v
		\end{equation*}
		for $1<p\leq2$, and
		\begin{equation*}
			\sup_{w\in\left(-\infty,0\right]}\abs{e^{uw}g_{I}\lr{w}}\leq\int_{\mathbb{R}}\abs{f\lr{u+Iv}}\mathrm{d}v
		\end{equation*}
		for $p=1$.
		
		For $1<p\leq2$, Minkowski's inequality implies that
		\begin{align*}
			&\int_{0}^{\infty}\int_{\mathbb{R}}\abs{f\lr{u+Iv}}^{p}\mathrm{d}v\mathrm{d}u\\
			\geq&\int_{0}^{\infty}\Lr{\int_{-\infty}^{0}\lr{\abs{e^{uw}g_{I}\lr{w}}^{p}}^{\frac{q}{p}}\mathrm{d}w}^{\frac{p}{q}}\mathrm{d}u\\
			\geq&\Lr{\int_{-\infty}^{0}\lr{\int_{0}^{\infty}\abs{e^{uw}g_{I}\lr{w}}^{p}\mathrm{d}u}^{\frac{q}{p}}\mathrm{d}w}^{\frac{p}{q}}\\
			=&\Lr{\int_{-\infty}^{0}\lr{\frac{1}{-pw}}^{\frac{q}{p}}\abs{g_{I}\lr{w}}^{q}\mathrm{d}w}^{\frac{p}{q}}
		\end{align*}
		and thus
		\begin{align*}
			&\Lr{\int_{-\infty}^{0}\lr{\frac{1}{-pw}}^{\frac{q}{p}}\abs{g_{I}\lr{w}}^{q}\mathrm{d}w}^{\frac{1}{q}}\\
			\leq&\lr{\int_{0}^{\infty}\int_{\mathbb{R}}\abs{f\lr{u+Iv}}^{p}\mathrm{d}v\mathrm{d}u}^{\frac{1}{p}}\\
			\leq&\norm{f}_{\mathcal{B}_{\text{slice}}^{p}\lr{\mathbb{R}_{+}^{n+1}}}<\infty.
		\end{align*}
		
		For $p=1$,
		\begin{equation*}
			\sup_{w\in\left(-\infty,0\right]}\frac{\abs{g_{I}\lr{w}}}{-w}\leq\int_{0}^{\infty}\int_{\mathbb{R}}\abs{f\lr{u+Iv}}\mathrm{d}v\mathrm{d}u\leq\norm{f}_{\mathcal{B}_{\text{slice}}^{1}\lr{\mathbb{R}_{+}^{n+1}}}<\infty.
		\end{equation*}
	\end{proof}
	In particular, when $p=2$, the converse of Theorem \ref{PW-B} holds.
	\begin{theorem}
		For every $\mathbf{x}\in\mathbb{R}_{+}^{n+1}$, let $g_{I_{\mathbf{x}}}\in L^{2}\lr{\lr{-\infty,0},\lr{-w}^{-1}\mathrm{d}w}$. Assume
		\begin{equation*}
			f\lr{\mathbf{x}}=\frac{1}{\sqrt{2\pi}}\int_{-\infty}^{0}e^{\mathbf{x}w}g_{I_{\mathbf{x}}}\lr{w}\mathrm{d}w.
		\end{equation*}
		Then $f\in\mathcal{B}_{\text{slice}}^{2}\lr{\mathbb{R}_{+}^{n+1}}$ satisfies that 
		\begin{equation*}
			\norm{f}_{\mathcal{B}_{\text{slice}}^{2}\lr{\mathbb{R}_{+}^{n+1}}}^{2}\leq\frac{1}{2}\int_{-\infty}^{0}\abs{g_{I_{\mathbf{x}}}\lr{w}}^{2}\frac{\mathrm{d}w}{-w}.
		\end{equation*}
	\end{theorem}
	\begin{proof}
		For every $I\in\mathbb{S}$, $u,v\in\mathbb{R}$ and $u>0$, by the Schwarz inequality, we get
		\begin{align*}
			\int_{-\infty}^{0}\Abs{e^{\lr{u+Iv}w}g_{I}\lr{w}}\mathrm{d}w&=\int_{-\infty}^{0}\Abs{e^{uw}\lr{-w}^{\frac{1}{2}}}\Abs{\lr{-w}^{-\frac{1}{2}}g_{I}\lr{w}}\mathrm{d}w\\
			&\leq\lr{\int_{-\infty}^{0}\Abs{e^{2uw}\lr{-w}}\mathrm{d}w}^{\frac{1}{2}}\lr{\int_{-\infty}^{0}\Abs{g_{I}\lr{w}}^{2}\lr{-w}^{-1}\mathrm{d}w}^{\frac{1}{2}}\\
			&<\infty.
		\end{align*}
		It implies that $f\lr{u+Iv}$ is holomorphic on $\Pi_{+,I}$ for each $I$. Hence, $f$ is left slice monogenic in $\mathbb{R}_{+}^{n+1}$.
		
		Next, using Plancherel's theorem, we obtain
		\begin{equation*}
			\int_{\mathbb{R}}\Abs{f\lr{u+Iv}}^{2}\mathrm{d}v=\int_{-\infty}^{0}e^{2uw}\Abs{g_{I}\lr{w}}^{2}\mathrm{d}w.
		\end{equation*}
		It follows that for every $J\in\mathbb{S}$
		\begin{align*}
			&\int_{0}^{\infty}\lr{\int_{\mathbb{R}}\abs{f\lr{u+Jv}}^{2}\mathrm{d}v}\mathrm{d}u\\
			\leq&\int_{0}^{\infty}\lr{\int_{\mathbb{R}}\abs{f\lr{u+Iv}}^{2}\mathrm{d}v}\mathrm{d}u\\
			=&\int_{-\infty}^{0}\lr{\int_{0}^{\infty}e^{2uw}\mathrm{d}u}\abs{g_{I}\lr{w}}^{2}\mathrm{d}w\\
			=&\frac{1}{2}\int_{-\infty}^{0}\abs{g_{I}\lr{w}}^{2}\lr{-w}^{-1}\mathrm{d}w.
		\end{align*}
		So, we have
		$$\sup_{J\in\mathbb{S}}\int_{0}^{\infty}\lr{\int_{\mathbb{R}}\abs{f\lr{u+Jv}}^{2}\mathrm{d}v}\mathrm{d}u<\infty.$$
		This implies that $f\in\mathcal{B}_{\text{slice}}^{2}\lr{\mathbb{R}_{+}^{n+1}}$.
	\end{proof}
	
	\bigskip
	%\subsection*{Author contributions.} All authors equally contributed to the elaboration of the paper. All authors read and approved the final manuscript.
	
	\subsection*{Acknowledgment. }W. X. Mai was supported by 
	the Science and Technology Development Fund, Macau SAR (No. 0133/2022/A, 0022/2023/ITP1).
	P. Dang was supported by the Science and Technology Development Fund, Macau SAR (No. 0030/2023/ITP1, 0067/2024/RIA1).

	\subsection*{Data availability statement. }
	Data sharing is not applicable to this article as no new data were created or analyzed in this study.
	
	\subsection*{Conflict of interest}
	The authors declared that they have no competing interest regarding this research work.

\end{document}